\documentclass[a4paper,11pt]{article}

\usepackage{fullpage} 
\usepackage{parskip} 
\usepackage{tikz} 
\usepackage{amsmath}
\usepackage{hyperref}
\usepackage{graphicx} 

\usepackage{enumerate}  
\usepackage{amsmath}    
\usepackage{amsthm}     
\usepackage{mathrsfs}   
\usepackage{amsfonts}
\usepackage{amssymb}
\usepackage{mathtools} 
\usepackage{verbatim}	
\usepackage{color}	
\usepackage{hyperref} 
\usepackage[T1]{fontenc}
\usepackage{subfig} 

\theoremstyle{plain}                
\newtheorem{thm}{Theorem}
\newtheorem{lem}[thm]{Lemma}        
\newtheorem{prop}[thm]{Proposition}
\newtheorem{cor}[thm]{Corollary}

\theoremstyle{definition}
\newtheorem{defn}[thm]{Definition}

\theoremstyle{remark}
\newtheorem{rem}[thm]{Remark}

\DeclareMathOperator{\Arg}{Arg}   

\providecommand{\abs}[1]{\lvert#1\rvert} 
\providecommand{\norm}[1]{\lVert#1\rVert}
\providecommand{\inner}[1]{\langle#1\rangle} 

\newcommand{\overbar}[1]{\mkern 1.5mu\overline{\mkern-1.5mu#1\mkern-1.5mu}\mkern 1.5mu} 
\newcommand{\vertiii}[1]{{\left\vert\kern-0.25ex\left\vert\kern-0.25ex\left\vert #1 
										 \right\vert\kern-0.25ex\right\vert\kern-0.25ex\right\vert}} 

\def\Re{\mathop{\rm Re}\nolimits}
\def\re{\mathop{\rm Re}\nolimits}

\def\exp{\mathop{\rm e}\nolimits}

\def\CC{\mathbb C}
\def\RR{\mathbb R}
\def\NN{\mathbb N}
\def\DD{\mathbb D}

\def\calA{\mathcal{A}}
\def\calB{\mathcal{B}}

\def\calX{\mathcal{X}}
\def\calL{\mathcal{L}}

\newlength\Colsep
\title{Admissibility of retarded diagonal systems with one-dimensional input space}
\author{Rafa{\l} Kapica\footnote{Faculty of Applied Mathematics, AGH University of Science and Technology, al. Mickiewicza 30, 30-059 Kraków}, \ Jonathan R. Partington\footnote{School of Mathematics, University of Leeds, LS2 9JT Leeds}\  \ and Rados{\l}aw Zawiski\footnote{Department of Automatic Control and Robotics, Silesian University of Technology, ul. Akademicka 16, 44-100 Gliwice}~\footnote{corresponding author}}
\date{ }

\begin{document}

\maketitle
\abstract
We investigate infinite-time admissibility of a control operator $B$ in a Hilbert space state-delayed dynamical system setting of the form $\dot{z}(t)=Az(t)+A_1 z(t-\tau)+Bu(t)$, where $A$ generates a diagonal $C_0$-semigroup, $A_1\in\calL(X)$ is also diagonal and $u\in L^2(0,\infty;\CC)$. Our approach is based on the Laplace embedding between $L^2$ and the Hardy space $H^2(\CC_+)$. The results are expressed in terms of the eigenvalues of $A$ and $A_1$ and  the sequence representing the control operator. 


\noindent {\bf 2020 Subject Classification:} 34K30, 34K35, 47D06, 93C23 

\section{Introduction}
State-delayed differential equations arise in many areas of applied mathematics, which is related to the fact that in the real world there is an inherent input-output delay in every physical system. Among sources of delay we have the spatial character of the system in relation to signal propagation, measurements processing or hatching time in biological systems, to name a few. Whenever the delay has a considerable influence on the outcome of the process it has to be incorporated into a process's mathematical model. Hence, an understanding of a state-delayed system, even in a linear case, plays a crucial role in the analysis and control of dynamical systems, particularly when the asymptotic behaviour is concerned. 

In order to cover a possibly large area of dynamical systems our analysis uses an abstract description. Hence the retarded state-delayed dynamical system we are interested in has an abstract representation given by
\begin{equation}\label{eq:retarded_non-autonomous_system}
\left\{\begin{array}{ll}
        \dot{z}(t)=Az(t)+A_{1}z(t-\tau)+Bu(t)\\
        z(0)=x\\
        z_0=f,
       \end{array}
\right.
\end{equation}
where the state space $X$ is a Hilbert space, $A:D(A)\subset X\to X$ is a closed, densely defined generator of a $C_0$-semigroup $(T(t))_{t\geq0}$ on $X$, $A_1\in\calL(X)$ and $0<\tau<\infty$ is a fixed delay (some discussions of the difficulties inherent in taking $A_1$ unbounded appear in Subsection \ref{subsec:Direct state-delayed diagonal systems}). The input function is $u\in L^2(0,\infty;\CC)$, $B$ is the control operator, the pair $x\in D(A)$ and $f\in L^2(-\tau,0;X)$ forms the initial condition. We also assume that $X$ possesses a sequence of normalized eigenvectors $(\phi_k)_{k\in\NN}$ forming a Riesz basis, with associated eigenvalues $(\lambda_k)_{k\in\NN}$. 

We analyse \eqref{eq:retarded_non-autonomous_system} from the perspective of infinite-time admissibility which, roughly speaking, asserts whether a solution $z$ of \eqref{eq:retarded_non-autonomous_system} follows a required type of trajectory. A more detailed description of admissibility requires an introduction of pivot duality and some related norm inequalities. For that reason we postpone it until Subsection~\ref{subsec:def_of_the_admissibility_problem}, where all these elements are already introduced for the setting within which we analyse \eqref{eq:retarded_non-autonomous_system}. 

With regard to previous admissibility results, necessary and sufficient conditions for infinite-time admissibility of $B$ in the undelayed case of \eqref{eq:retarded_non-autonomous_system}, under an assumption of diagonal generator $(A,D(A))$, were analysed e.g. using Carleson measures e.g.  in \cite{Ho_Russell_1983,Ho_Russell_1983_err,Weiss_1988}. 
Those results were extended to normal semigroups \cite{Weiss_1999}, then generalized to the case when $u\in L^2(0,\infty;t^{\alpha}dt)$ for $\alpha\in(-1,0)$ in \cite{Wynn_2010} and further to the case $u\in L^2(0,\infty;w(t)dt)$ in \cite{Jacob_Partington_Pott_2013,Jacob_Partington_Pott_2014}. 
For a thorough presentation of admissibility results, not restricted to diagonal systems, for the undelayed case we refer the reader to \cite{Jacob_Partington_2004} and a rich list of references therein. 

For the delayed case, in contrast to the undelayed one, a different setting is required. Some of the first studies in such setting are \cite{Grabowski_Callier_1996} and \cite{Engel_1999}, and these form a basis for \cite{Batkai_Piazzera}. In this article we follow the latter one in developing a setting for admissibility analysis. We also build on \cite{Partington_Zawiski_2019} where a similar setting was used to present admissibility results for a simplified version of \eqref{eq:retarded_non-autonomous_system}, that is with a diagonal generator $(A,D(A))$ with the delay in its argument (see the Examples section below). 

In fact, as the system analysed in \cite{Partington_Zawiski_2019} is a special case of \eqref{eq:retarded_non-autonomous_system}, the results presented here contain those of \cite{Partington_Zawiski_2019}. The most important drawback of results in \cite{Partington_Zawiski_2019} is that the conditions leading to sufficiency for infinite-time admissibility there imply also that the semigroup generator is bounded. Thus, to obtain some results for unbounded generators one is forced to go though the so-called reciprocal systems. Results presented below are free from such limitation and can be applied to unbounded diagonal generators directly, as shown in the Examples section.

This paper is organised as follows. Section~\ref{sec:2} defines the notation and provides preliminary results. These include a general delayed equation setting, which is applied later to the problem of our interest and the problem of infinite-time admissibility. Section~\ref{sec:3} shows how the general setting looks for a particular case of retarded diagonal case. It then shows a component-wise analysis of infinite-time admissibility and provides results for the complete system. Section~\ref{sec:4} gives examples.

\section{Preliminaries}\label{sec:2}

In this paper we use Sobolev spaces (see e.g. \cite[Chapter 5]{Evans})
$W^{1,2}(J,X):=\{f\in L^2(J,X):\frac{d}{dt}f\in L^2(J,X)\}$ and $W_{0}^{1,2}(J,X):=\{f\in W^{1,2}(J,X):\ f(\partial J)=0\}$, where $\frac{d}{dt}f$ is a weak derivative of $f$ and $J$ is an interval with boundary $\partial J$. 

For any $\alpha\in\mathbb{R}$ we denote the following half-planes
\[
\underleftarrow{\CC}_{\alpha}:=\{s\in\CC:\re s<\alpha\},\quad\underrightarrow{\CC}_{\alpha}:=\{s\in\CC:\re s>\alpha\},
\]
with a simplification for two special cases, namely $\CC_{-}:=\underleftarrow{\CC}_{0}$ and $\CC_{+}:=\underrightarrow{\CC}_{0}$.
We make use of the Hardy space $H^2(\CC_+)$ that consists of all analytic functions $f:\CC_+\rightarrow\CC$ for which 
\begin{equation}\label{eqn defn Hardy space element}
\sup_{\alpha>0}\int_{-\infty}^{\infty}\abs{f(\alpha+i\omega)}^2 \, d\omega<\infty.
\end{equation}
If $f\in H^2(\CC_+)$ then for a.e. $\omega\in\RR$ the limit
\begin{equation}\label{eqn defn boundary trace}
f^*(i\omega)=\lim_{\alpha\downarrow 0}f(\alpha+i\omega)
\end{equation}
exists and defines a function  $f^*\in L^2(i\RR)$ called the \textit{boundary trace} of $f$.  Using boundary traces $H^2(\CC_+)$ is made into a Hilbert space with the inner product defined as
\begin{equation}\label{eq:inner_product_on_H2(C+)}
\inner{f,g}_{H^2(\CC_+)}:=\inner{f^*,g^*}_{L^2(i\RR)}:=\frac{1}{2\pi}\int_{-\infty}^{+\infty}f^*(i\omega)\bar{g}^*(i\omega) \, d\omega\quad\forall f,g\in H^2(\CC_+).
\end{equation}
For more information about Hardy spaces see \cite{Partington_1988}, \cite{Garnett} or \cite{Koosis}. We also make use of the Paley--Wiener Theorem (see \cite[Chapter 19]{Rudin_1987} for the scalar version or \cite[Theorem 1.8.3]{Arendt_et_al} for the vector-valued one)
\begin{thm}[Paley--Wiener]\label{thm:Paley-Wiener}
Let $Y$ be a Hilbert space. Then the Laplace transform $\mathcal{L}:L^2(0,\infty;Y)\rightarrow H^2(\CC_+;Y)$ is an isometric isomorphism.
\end{thm}

\subsection{The delayed equation setting}

We follow a general setting for a state-delayed system from \cite[Chapter 3.1]{Batkai_Piazzera}, described for a diagonal case also in \cite{Partington_Zawiski_2019}. And so, to include the influence of the delay we extend the state space of \eqref{eq:retarded_non-autonomous_system}. To that end consider a trajectory of \eqref{eq:retarded_non-autonomous_system} given by  $z:[-\tau,\infty)\rightarrow X$. For each $t\geq0$ we call $z_{t}:[-\tau,0]\rightarrow X$, $z_{t}(\sigma):=z(t+\sigma)$ a \textit{history segment} with respect to $t\geq0$. With history segments we consider a so-called  \textit{history function} of $z$ denoted by $h_{z}:[0,\infty)\rightarrow L^2(-\tau,0;X)$, $h_z(t):=z_{t}$. In \cite[Lemma 3.4]{Batkai_Piazzera} we find the following

\begin{prop}\label{prop:history_function_derivative}
Let $1\leq p<\infty$ and $z\in W_{loc}^{1,p}(-\tau,\infty;X)$. Then the history function $h_z:t\rightarrow z_t$ of $z$ is continuously differentiable from $\mathbb{R}_+$ into $L^p(-\tau,0;X)$ with derivative 
\[
\frac{\partial}{\partial t}h_{z}(t)=\frac{\partial}{\partial\sigma}z_{t}.
\]
\end{prop}
To remain in the Hilbert space setting we limit ourselves to $p=2$ and take 
\begin{equation}\label{eq:extended_state_space_def}
\calX:=X\times L^2(-\tau,0;X)
\end{equation}
as the aforementioned state space extension with an inner product
\begin{equation}\label{eq:inner_product_on_Cartesian_product_def}
 \bigg\langle\binom{x}{f},\binom{y}{g}\bigg\rangle_{\mathcal{X}}:=\langle x,y\rangle_{X}+\langle f,g\rangle_{L^2(-\tau,0;X)}.
\end{equation}
Then $(\mathcal{X},\|\cdot\|_{\mathcal{X}})$ becomes a Hilbert space with the norm $\|\binom{x}{f}\|_{\mathcal{X}}^2=\|x\|_{X}^{2}+\|f\|_{L^2}^{2}$.  We assume that a linear and bounded \textit{delay operator} $\Psi:W^{1,2}(-\tau,0;X)\to X$ acts on history segments $z_t$ and thus consider \eqref{eq:retarded_non-autonomous_system} in the form
\begin{equation}\label{eq:delay_non-autonomous_diff_eq}
\left\{\begin{array}{ll}
        \dot{z}(t)=Az(t)+\Psi z_t+Bu(t)\\
        z(0)=x,        \\
        z_0=f,		\\
       \end{array}
\right.
\end{equation}
where the pair $x\in D(A)$ and $f\in L^2(-\tau,0;X)$ forms an initial condition. 
A particular choice of $\Psi$ can be found in \eqref{eq:delay_operator_def} below.
Due to Proposition~\ref{prop:history_function_derivative}, system \eqref{eq:delay_non-autonomous_diff_eq} may be written as an abstract Cauchy problem
\begin{equation}\label{eq:non-autonomous_abstract_Cauchy_problem_def}
\left\{\begin{array}{ll}
        \dot{v}(t)=\mathcal{A}v(t)+\mathcal{B}u(t)\\
        v(0)=\binom{x}{f},        \\
       \end{array}
\right.
\end{equation} 
where $v:[0,\infty)\ni t\mapsto\binom{z(t)}{z_t}\in\mathcal{X}$ 
and $\mathcal{A}$ is a linear operator on $D(\calA)\subset\mathcal{X}$, where
\begin{equation}\label{eq:abstract_A_domain_def}
D(\mathcal{A}):=\bigg\{\binom{x}{f}\in D(A)\times W^{1,2}(-\tau,0;X):\ f(0)=x\bigg\},
\end{equation}
\begin{equation}\label{eq:abstract_A_def}
\mathcal{A}:=\left(\begin{array}{cc} 
								A & \Psi \\ 
								0 & \frac{d}{d\sigma}	
						\end{array}\right),
\end{equation}
and the control operator is $\mathcal{B}=\binom{B}{0}$. Operator $(\mathcal{A},D(\mathcal{A}))$ is closed and densely defined on $\mathcal{X}$ \cite[Lemma 3.6]{Batkai_Piazzera}. Note that up to this moment we do not need to know more about $\Psi$.

Concerning the resolvent of $(\mathcal{A},D(\mathcal{A}))$, let 
\[
A_{0}:=\frac{d}{d\sigma},\qquad D(A_{0})=\{z\in W^{1,2}(-\tau,0;X):z(0)=0\},
\]
be the generator of a nilpotent left shift semigroup on $L^2(-\tau,0;X)$. For $s\in\mathbb{C}$ define $\epsilon_{s}:[-\tau,0]\rightarrow\mathbb{C}$, $\epsilon_{s}(\sigma):=e^{s \sigma}$. Define also $\Psi_{s}\in\mathcal{L}(D(A),X)$, $\Psi_{s}x:=\Psi(\epsilon_{s}(\cdot)x)$. Then \cite[Proposition 3.19]{Batkai_Piazzera} provides
\begin{prop}\label{prop abstract A resolvent operator}
For $s\in\mathbb{C}$ and for all $1\leq p<\infty$ we have 
\[
s\in\rho(\mathcal{A}) \quad \hbox{if and only if} \quad s\in\rho(A+\Psi_{s}).
\]
Moreover, for $s\in\rho(\mathcal{A})$ the resolvent $R(s,\mathcal{A})$ is given by
\begin{equation}\label{eqn relosvent of abstract A}
R(s,\mathcal{A})=\left(\begin{array}{ll} 
	R(s,A+\Psi_{s}) & R(s,A+\Psi_{s})\Psi R(s,A_0) \\ 
	\epsilon_{s}R(s,A+\Psi_{s}) & (\epsilon_{s}R(s,A+\Psi_{s})\Psi+I)	R(s,A_0)
								  \end{array}\right).
\end{equation}

\end{prop}

In the sequel we make use of Sobolev towers, also known as a duality with a pivot (see \cite[Chapter 2]{Tucsnak_Weiss} or \cite[Chapter II.5]{Engel_Nagel}). To this end we have 
\begin{defn}\label{def:Sobolev_tower}
Let $\beta\in\rho(A)$ and denote $(X_1,\norm{\cdot}_1):=(D(A),\norm{\cdot}_1)$ with $\norm{x}_1:=\norm{(\beta I-A)x}\ (x\in D(A))$. Similarly, we set $\norm{x}_{-1}:=\norm{(\beta I-A)^{-1}x}\ (x\in X)$. Then the space $(X_{-1},\norm{\cdot}_{-1})$ denotes the completion of $X$ under the norm $\norm{\cdot}_{-1}$. For $t\geq0$ we define $T_{-1}(t)$ as the continuous extension of $T(t)$ to the space  $(X_{-1},\norm{\cdot}_{-1})$. 
\end{defn}
The adjoint generator plays an important role in the pivot duality setting. Thus we take
\begin{defn}\label{def:adjoint_of_a_generator}
Let $A:D(A)\rightarrow X$ be a densely defined operator. The \textit{adjoint} of $(A,D(A))$, denoted $(A^*,D(A^*))$, is defined on \index{Latin@\textbf{Latin symbols}!A@$A^*$}
\begin{equation}\label{eq:adjoint_of_a_generator domain}
D(A^*):=\{y\in X: \hbox{the functional}\ X\ni x\mapsto\inner{Ax,y}\ \hbox{is bounded}\}.
\end{equation} 
Since $D(A)$ is dense in $X$ the functional in \eqref{eq:adjoint_of_a_generator domain} has a unique bounded extension to $X$. By the Riesz representation theorem there exists a unique $w\in X$ such that $\inner{Ax,y}=\inner{x,w}$. Then we define $A^*y:=w$ so that 
\begin{equation}\label{eq:adjoint_of_a_generator_form}
\inner{Ax,y}=\inner{x,A^*y}\qquad \forall\ x\in D(A)\quad \forall\ y\in D(A^*).
\end{equation}
\end{defn} 
We have the following (see \cite[Prop. 2.10.2]{Tucsnak_Weiss})
\begin{prop}\label{prop:Sobolev_tower_adjoint}
With the notation of Definition~\ref{def:Sobolev_tower}  let $(A^*,D(A^*))$ be the adjoint of $(A,D(A))$. Then $\overbar{\beta}\in\rho(A^*)$, $(X_1^d,\norm{\cdot}_1^d):=(D(A^*),\norm{\cdot}_1^d)$ with $\norm{x}_1^d:=\norm{(\overbar{\beta}I-A^*)x}\ (x\in D(A^*))$ is a Hilbert space and $X_{-1}$ is the dual of $X_1^d$ with respect to the pivot space $X$, that is $X_{-1}=(D(A^*))'$.
\end{prop}
Much of our reasoning is justified by the following Proposition,  which we  include here for the reader's convenience (for more details see \cite[Chapter II.5]{Engel_Nagel} or \cite[Chapter 2.10]{Tucsnak_Weiss}).
\begin{prop}\label{prop:Hilbert_rigged_space}
With the notation of Definition~\ref{def:Sobolev_tower} we have the following
\begin{itemize}
 \item[(i)] The spaces $(X_1,\norm{\cdot}_1)$ and $(X_{-1},\norm{\cdot}_{-1})$ are independent of the choice of $\beta\in\rho(A)$.
 \item[(ii)] $(T_1(t))_{t\geq0}$ is a $C_0$-semigroup on the Banach space $(X_1,\norm{\cdot}_1)$ and we have $\norm{T_1(t)}_1=\norm{T(t)}$ for all $t\geq0$.
 \item[(iii)] $(T_{-1}(t))_{t\geq0}$ is a $C_0$-semigroup on the Banach space $(X_{-1},\norm{\cdot}_{-1})$ and  $\norm{T_{-1}(t)}_{-1}=\norm{T(t)}$ for all $t\geq0$.
\end{itemize}
\end{prop}
In the sequel, we denote the restriction (extension) of $T(t)$ described in Definition~\ref{def:Sobolev_tower} by the same symbol $T(t)$, since this is unlikely to lead to confusions.

In the sequel we also use the following result by Miyadera and Voigt \cite[Corollaries III.3.15 and 3.16]{Engel_Nagel}, that gives sufficient conditions for a perturbed generator to remain a generator of a $C_0$-semigroup.
\begin{prop}\label{prop:Miyadera-Voigt_sufficient_condition_of_well-posedness}
Let $(A,D(A))$ be the generator of a strongly continuous semigroup $\big(T(t)\big)_{t\geq0}$ on a Banach space $X$ and let ${P\in\mathcal{L}(X_1,X)}$ be a perturbation which satisfies
\begin{equation}\label{eq:condition_on_perturbation_for_well-posedness}
\int_{0}^{t_0}\norm{PT(r)x}dr\leq q\Vert x\Vert\qquad\forall x\in D(A)
\end{equation}
for some $t_0>0$ and $0\leq q<1$. Then the sum $A+P$ with domain $D(A+P):=D(A)$ generates a strongly continuous semigroup $(S(t))_{t\geq0}$ on $X$. Moreover, for all $t\geq0$ the $C_0$-semigroup $(S(t))_{t\geq0}$ satisfies
\begin{equation}\label{eq:perturbed_semigroup_in_integral_form}
 S(t)x=T(t)x+\int_{0}^{t}S(s)PT(t-s)x\,ds \qquad\forall x\in D(A).
\end{equation}
\end{prop}

\subsection{The admissibility problem}\label{subsec:def_of_the_admissibility_problem}
The basic object in the formulation of admissibility problem is a linear system and its mild solution
\begin{equation}\label{eqn basic object}
 \dot{x}(t)=Ax(t)+Bu(t);\quad x(t)=T(t)x_0+\int_{0}^{t}T(t-s)Bu(s) \, ds,
\end{equation}
where $x:[0,\infty)\rightarrow X$, $u\in V$, where $V$ is a normed space of measurable functions from $[0,\infty)$ to $U$ and $B$ is a \textit{control operator}; $x_0\in X$ is an initial state.

In many practical examples the control operator $B$ is unbounded, hence \eqref{eqn basic object} is viewed on an extrapolation space $X_{-1}\supset X$ where $B\in\mathcal{L}(U,X_{-1})$. Introduction of $X_{-1}$, however, comes at a price of physical interpretation of the solution. To be more precise, a dynamical system expressed by \eqref{eqn basic object} describes a physical system where one can assign a physical meaning to $X$, with the use of which the modelling is performed. That is not always true for $X_{-1}$. We would then like to study those control operators $B$ for which the (mild) solution is a continuous $X$-valued function that carries a physical meaning. In a rigorous way, to ensure that the state $x(t)$ lies in $X$ it is sufficient that $\int_{0}^{t}T_{-1}(t-s)Bu(s) \, ds\in X$ for all inputs $u\in V$. 
\begin{defn}\label{def:forcing_operator}
Let $B\in\calL(U,X_{-1})$ and $t\geq0$. The \textit{forcing operator} $\Phi_{t}\in\calL(V,X_{-1})$ is given by \index{Greek@\textbf{Greek symbols}!u@$\Phi_{t}(\cdot)$}
\begin{equation}\label{eq:forcing_operator}
\Phi_{t}(u):=\int_{0}^{t}T(t-\sigma)Bu(\sigma)\,d\sigma.
\end{equation}
\end{defn}

Put differently, we have 
\begin{defn}\label{def:infinite_time_admissibility_of_B}
 The control operator $B\in\mathcal{L}(U,X_{-1})$ is called
 \begin{enumerate}[(i)]
 \item \textit{finite-time admissible} for $\big(T(t)\big)_{t\geq0}$ on a Hilbert space $X$ if for each $t>0$ there is a constant $K_t$ such that 
\begin{equation}\label{eq:finite-time_admissibility_by_norm_inequality_def}
\norm{\Phi_{t}(u)}_X\leq K_t\norm{u}_{V}\quad\forall u\in V;
\end{equation}
 \item  \textit{infinite-time admissible} for $(T(t))_{t\geq0}$ if there is a constant $K\geq0$ such that
\begin{equation}\label{eq:infinite_time_admissibility_def}
\norm{\Phi_t}_{\calL(V,X)}\leq K\qquad \forall t\geq0.
\end{equation}
\end{enumerate} 
\end{defn}
For the infinite-time admissibility it is convenient to define a different version of the forcing operator, namely $\Phi_\infty:L^2(0,\infty;U)\to X_{-1}$,
\begin{equation}\label{eq:forcing_operator_infty_def}
\Phi_{\infty}(u):=\int_{0}^{\infty}T(t)Bu(t) \, dt.
\end{equation}
The infinite-time admissibility of $B$ follows then from the boundedness of $\Phi_\infty$ in \eqref{eq:forcing_operator_infty_def} taken as an operator from $L^2(0,\infty;U)$ to $X$. For a more detailed discussion concerning infinite-time admissibility see also \cite{Jacob_Partington_2004} and \cite{Tucsnak_Weiss} with references therein.

\section{The setting of retarded diagonal systems}\label{sec:3}
We begin with a general setting of the previous section expressed by \eqref{eq:non-autonomous_abstract_Cauchy_problem_def} with elements defined there. Then, consecutively specifying these elements, we reach a description of a concrete case of a retarded diagonal system. 

Let the delay operator $\Psi$ be a point evaluation i.e. define $\Psi\in\mathcal{L}(W^{1,2}(-\tau,0;X),X)$ as
\begin{equation}\label{eq:delay_operator_def}
\Psi(f):=A_{1}f(-\tau),
\end{equation}
where boundedness of $\Psi$ results from continuous embedding of $W^{1,2}(-\tau,0;X)$  in $C([-\tau,0],X)$ (see e.g. \cite[Theorem 8.8]{Brezis}, \cite[Theorem III.4.10.2]{Amann} or \cite[Chapter 5.9.2]{Evans}).

With the delay operator given by \eqref{eq:delay_operator_def} we are in a position to describe pivot duality for $\mathcal{X}$ given by \eqref{eq:extended_state_space_def} with $(\calA,D(\calA))$ given by \eqref{eq:abstract_A_def}-\eqref{eq:abstract_A_domain_def} and with $\mathcal{B}=\binom{B}{0}$. Then, using the pivot duality, we consider \eqref{eq:non-autonomous_abstract_Cauchy_problem_def} on the completion space $\mathcal{X}_{-1}$ where the control operator $\mathcal{B}\in\mathcal{L}(U,\mathcal{X}_{-1})$. To write  explicitly all the elements of the pivot duality setting we need to determine the adjoint $(\calA^*,D(\calA^*))$ operator (see Proposition~\ref{prop:Sobolev_tower_adjoint}).

\begin{prop}\label{prop:adjoint_of_abstract_A}
Let $X$, $(A,D(A))$ and $A_1$ be as in \eqref{eq:retarded_non-autonomous_system} and $(\mathcal{A},D(\mathcal{A}))$ be defined by \eqref{eq:abstract_A_def}-- \eqref{eq:abstract_A_domain_def} with $\Psi$ given by \eqref{eq:delay_operator_def}. Then $(\calA^*,D(\calA^*))$, the adjoint of $(\calA,D(\calA))$, is given by

\begin{equation}\label{eq:adjoint_of_abstract_A_domain}
D(\mathcal{A}^*)=\bigg\{\binom{y}{g}\in D(A^*)\times W^{1,2}(-\tau,0;X):A_{1}^*y=g(-\tau)\bigg\},
\end{equation}
\begin{equation}\label{eq:adjoint_of_abstract_A_def}
\mathcal{A}^*\binom{y}{g}=\binom{A^*y+g(0)}{-\frac{d}{d\sigma}g},
\end{equation}
where $(A^*,D(A^*))$ is the adjoint of $(A,D(A))$ and $A_1^*$ is the adjoint of $A_1$.
\end{prop}
\begin{proof} Let $F$ be the set defined as the right hand side of \eqref{eq:adjoint_of_abstract_A_domain}. 
To show that $D(\calA^*)\subset F$ we adapt the approach from \cite{Kappel_1986}. Let $v=\binom{f(0)}{f}\in D(\calA)$, $w=\binom{y}{g}\in D(\calA^*)$ and let
\[
\calA^*w=\binom{(\calA^*w)^0}{(\calA^*w)^1}.
\]
By \eqref{eq:abstract_A_def}, \eqref{eq:abstract_A_domain_def} and the adjoint Definition \ref{def:adjoint_of_a_generator} we get
\begin{align}\label{eq:adjoint_of_abstract_A_aux_1}
\begin{split}
\inner{\calA^* w,v}_\calX&=\big\langle (\calA^*w)^0, f(0)\big\rangle_X+\int_{-\tau}^{0}\big\langle (\calA^*w)^1(\sigma),f(\sigma)\big\rangle_X\,d\sigma\\
 &=\big\langle y,A f(0)\big\rangle_X+\big\langle y,A_1 f(-\tau)\big\rangle_X+\int_{-\tau}^{0}\big\langle g(\sigma),\frac{d}{d\sigma}f\big\rangle_X\,d\sigma,
 \end{split}
\end{align}
and boundedness of the above for every $v\in D(\calA)$ implies that $y\in D(A^*)$. Observe also that 
\begin{align}\label{eq:adjoint_of_abstract_A_aux_2}
\begin{split} 
&\int_{-\tau}^{0}\big\langle (\calA^*w)^1(\sigma),f(\sigma)\big\rangle_X\,d\sigma=
\int_{-\tau}^{0}\Big\langle (\calA^*w)^1(\sigma),f(0)-\int_{\sigma}^{0}\frac{d}{d\xi}f(\xi)\,d\xi \Big\rangle_X\,d\sigma\\\
&=\int_{-\tau}^{0}\big\langle (\calA^*w)^1(\sigma),f(0)\big\rangle_X\,d\sigma-\int_{-\tau}^{0}\Big\langle (\calA^*w)^1(\sigma),\int_{\sigma}^{0}\frac{d}{d\xi}f(\xi)\,d\xi \Big\rangle_X\,d\sigma\\
&=\int_{-\tau}^{0}\big\langle (\calA^*w)^1(\sigma),f(0)\big\rangle_X\,d\sigma-\int_{-\tau}^{0}\int_{-\tau}^{\xi}\Big\langle (\calA^*w)^1(\sigma),\frac{d}{d\xi}f(\xi)\Big\rangle_X\,d\sigma\,d\xi\\
&=\Big\langle \int_{-\tau}^{0}(\calA^*w)^1(\sigma)\,d\sigma,f(0)\Big\rangle_X-\int_{-\tau}^{0}\Big\langle \int_{-\tau}^{\xi}(\calA^*w)^1(\sigma)\,d\sigma,\frac{d}{d\xi}f(\xi)\Big\rangle_X\,d\xi.
\end{split}
\end{align}
Putting the result of \eqref{eq:adjoint_of_abstract_A_aux_2} into \eqref{eq:adjoint_of_abstract_A_aux_1} and rearranging gives that for every $v\in D(\calA)$
\begin{align}
\begin{split}
&\big\langle (\calA^*w)^0+ \int_{-\tau}^{0}(\calA^*w)^1(\sigma)\,d\sigma-A^*y-A_1^*y ,f(0)\big\rangle_X\\
&=\int_{-\tau}^{0}\Big\langle \int_{-\tau}^{\sigma}(\calA^*w)^1(\xi)\,d\xi-A_1^*y+g(\sigma),\frac{d}{d\sigma}f(\sigma)\Big\rangle_X\,d\sigma,
\end{split}
\end{align}
where we used the fact that $f(-\tau)=f(0)-\int_{-\tau}^{0}\frac{d}{d\sigma}f(\sigma)\,d\sigma$. As for every constant $f:[-\tau,0]\to D(A)$ we have $\binom{f(0)}{f}\in D(\calA)$, there is 
\begin{equation}\label{eq:adjoint_of_abstract_A_aux_3}
(\calA^*w)^0=A^*y+A_1^*y-\int_{-\tau}^{0}(\calA^*w)^1(\sigma)\,d\sigma,
\end{equation} 
and then
\begin{equation}\label{eq:adjoint_of_abstract_A_aux_4}
g(\sigma)=A_1^*y-\int_{-\tau}^{\sigma}(\calA^*w)^1(\xi)\,d\xi, \quad \forall \sigma\in[-\tau,0].
\end{equation}
Equation \eqref{eq:adjoint_of_abstract_A_aux_4} shows that $w=\binom{y}{g}\in D(\calA^*)$ implies $g\in W^{1,2}(-\tau,0;X)$. Taking the limits gives
\begin{equation}\label{eq:djoint_of_abstract_A_aux_5}
g(-\tau)=A^*_{1}y
\end{equation} 
and
\begin{equation}\label{eq:djoint_of_abstract_A_aux_6}
(\calA^*w)^0=A^*y+g(0).
\end{equation}
Differentiating \eqref{eq:adjoint_of_abstract_A_aux_4} with respect to $\sigma$ we also have
\begin{equation}\label{eq:djoint_of_abstract_A_aux_7}
(\calA^*w)^1=-\frac{d}{d\sigma}g.
\end{equation}
To show that $D(\calA^*)\supset F$ let $w=\binom{y}{g}\in F$ and $v=\binom{x}{f}=\binom{f(0)}{f}\in D(\calA)$. By \eqref{eq:adjoint_of_a_generator_form} we need to show that $\inner{\calA^* w,v}_\calX=\inner{w,\calA v}_\calX$, where $\calA^*w$ we take as given by \eqref{eq:adjoint_of_abstract_A_def}. We have 
\begin{align*}
\inner{\calA^* w,v}_\calX&=\bigg\langle\binom{A^*y+g(0)}{-\frac{d}{d\sigma}g},\binom{f(0)}{f}\bigg\rangle_\calX
=\big\langle A^*y+g(0),f(0)\big\rangle_X+\int_{-\tau}^{0}\big\langle -\frac{d}{d\sigma}g,f\big\rangle_X\,d\sigma\\
&=\big\langle A^*y,f(0)\big\rangle_X+\big\langle g(0),f(0)\big\rangle_X+\big\langle -g,f\big\rangle_X\bigg|_{-\tau}^{0}-\int_{-\tau}^{0}\big\langle -g,\frac{d}{d\sigma}f\big\rangle_X\,d\sigma\\
&=\big\langle A^*y,f(0)\big\rangle_X+\big\langle g(-\tau),f(-\tau)\big\rangle_X+\big\langle g,\frac{d}{d\sigma}f\big\rangle_{L^2}\\
&=\big\langle A^*y,f(0)\big\rangle_X+\big\langle A_{1}^*y,f(-\tau)\big\rangle_X+\big\langle g,\frac{d}{d\sigma}f\big\rangle_{L^2}\\
&=\big\langle y,A f(0)\big\rangle_X+\big\langle y,A_1 f(-\tau)\big\rangle_X+\big\langle g,\frac{d}{d\sigma}f\big\rangle_{L^2}=\inner{w,\calA v}_\calX.
\end{align*}
\end{proof}
Denoting $D(\calA^*)'$ for the dual to $D(\calA^*)$ with respect to the pivot space $\calX$, by Proposition~\ref{prop:Sobolev_tower_adjoint} we have
\begin{equation}\label{eq:extended_space_X_-1}
 \mathcal{X}_{-1}=D(\calA^*)'.
\end{equation}

System \eqref{eq:non-autonomous_abstract_Cauchy_problem_def} represents an abstract Cauchy problem, which is well-posed if and only if $(\calA,D(\calA))$ generates a $C_0$-semigroup on $\calX$. To show that this is the case we use a perturbation approach. We represent  $\mathcal{A}=\mathcal{A}_0+\mathcal{A}_{\Psi}$, where 
\begin{equation}\label{eq:A_0_def}
\mathcal{A}_0:=\left(\begin{array}{cc} 
								A & 0 \\ 
								0 & \frac{d}{d\sigma}	
						\end{array}\right),
\end{equation}
with domain $D(\mathcal{A}_0)=D(\mathcal{A})$ and
\begin{equation}\label{eq:A_psi_def}
\mathcal{A}_{\Psi}:=\left(\begin{array}{cc} 
								0 & \Psi \\ 
								0 & 0	
						\end{array}\right),
\end{equation}
where $\mathcal{A}_{\Psi}\in\mathcal{L}\big(X\times W^{1,2}(-\tau,0;X),\mathcal{X}\big)$.  
The following proposition \cite[Theorem 3.25]{Batkai_Piazzera} gives a necessary and sufficient condition for the unperturbed part $(\mathcal{A}_0,D(\mathcal{A}_0))$ to generate a $C_0$-semigroup on $\calX$.
\begin{prop}\label{prop:abstract_A_0_as_generator}
Let $X$ be a Banach space. The following are equivalent:
	\begin{itemize}
		\item[(i)] The operator $(A,D(A))$ generates a strongly continuous semigroup\\ $\big(T(t)\big)_{t\geq0}$ on $X$.
		\item[(ii)] The operator $(\mathcal{A}_0,D(\mathcal{A}_0))$ generates a strongly continuous semigroup\\ $\big(\mathcal{T}_0(t)\big)_{t\geq0}$ on $X\times L^p(-\tau,0;X)$ for all $1\leq p<\infty$.
		\item[(iii)] The operator $(\mathcal{A}_0,D(\mathcal{A}_0))$ generates a strongly continuous semigroup\\ $\big(\mathcal{T}_0(t)\big)_{t\geq0}$ on $X\times L^p(-\tau,0;X)$ for one $1\leq p<\infty$.
	\end{itemize}
	The $C_0$-semigroup $\big(\mathcal{T}_0(t)\big)_{t\geq0}$ is given by 
	\begin{equation}\label{eq:T_0_semigroup_def}
		\mathcal{T}_{0}(t):=\left(\begin{array}{cc} 
									T(t) & 0 \\ 
									S_{t} & S_{0}(t)	
									\end{array}\right)\qquad \forall t\geq0,
	\end{equation}
	where $\big(S_{0}(t)\big)_{t\geq0}$ is the nilpotent left shift $C_0$-semigroup on $L^p(-\tau,0;X)$, 
	\begin{equation}
	S_{0}(t)f(s):=\left\{\begin{array}{ll}
        								f(s+t)	&	if\ s+t\in[-\tau,0],\\
        								0		&	else\\
       									\end{array}\right.
	\end{equation}
	and $S_{t}:X\rightarrow L^p(-\tau,0;X)$, 
	\begin{equation}
	(S_{t}x)(s):=\left\{\begin{array}{ll}
        								T(s+t)x	&	if\ -t<s\leq0,\\
        								0					&	if -\tau\leq s\leq-t.\\
       									\end{array}\right.
	\end{equation}
\end{prop}

Proposition~\ref{prop:Miyadera-Voigt_sufficient_condition_of_well-posedness} provides now a sufficient condition for the perturbation $(\mathcal{A}_\Psi,D(\mathcal{A}))$ such that $\mathcal{A}=\mathcal{A}_0+\mathcal{A}_{\Psi}$ is a generator, as given by the following

\begin{prop}\label{prop:non_homogeneous_abstract_Cauchy_problem_well-posedness}
Operator $(\mathcal{A},D(\mathcal{A}))$ generates a $C_0$-semigroup $\big(\mathcal{T}(t)\big)_{t\geq0}$ on $\calX$.
\end{prop}
\begin{proof}
We use Proposition \ref{prop:Miyadera-Voigt_sufficient_condition_of_well-posedness} with $(\mathcal{A},D(\mathcal{A}))$ given by \eqref{eq:abstract_A_def}-- \eqref{eq:abstract_A_domain_def} and represented as sum of \eqref{eq:A_0_def} and \eqref{eq:A_psi_def} with $\Psi$ given by \eqref{eq:delay_operator_def}. Thus a sufficient condition for $(\mathcal{A},D(\mathcal{A}))$ to be a generator of a strongly continuous semigroup $\big(\mathcal{T}(t)\big)_{t\geq0}$ on $\calX$ is that the perturbation $\mathcal{A}_{\Psi}\in\mathcal{L}\big(X\times W^{1,2}(-\tau,0;X),\mathcal{X}\big)$ given by \eqref{eq:A_psi_def} satisfies 
\[
 \int_{0}^{t_0}\Big\|\mathcal{A}_\Psi\mathcal{T}_0(r)v\Big\|_\mathcal{X}dr\leq q\norm{v}_\mathcal{X}\qquad \forall v\in D(\mathcal{A}_0)
\]
for some $t_0>0$ and $0\leq q<1$.
 
Let $\binom{x}{f}\in D(\mathcal{A}_0)$ and let $0< t<1$. Then, using the notation of  Proposition \ref{prop:abstract_A_0_as_generator} and defining $M:=\max\big\{\sup_{s\in[0,1]}\norm{T(s)},1\big\}$ we have
  \begin{align*}
  &\int_{0}^{t}\Big\|\mathcal{A}_\Psi\mathcal{T}_0(r)v\Big\|_\mathcal{X}\,dr=\int_{0}^{t}\norm{\Psi(S_{r}x+S_{0}(r)f)}_X\,dr\\
  &=\int_{0}^{t}\norm{A_1(S_{r}x)(-\tau)+A_1S_{0}(r)f(-\tau)}_X\,dr\\
  &\leq\norm{A_1}\int_{0}^{t}\norm{T(-\tau+r)x}_X\,dr+\norm{A_1}\int_{0}^{t}\norm{f(-\tau+r)}_X\,dr\\
  &=\norm{A_1}\int_{-\tau}^{-\tau+t}\norm{T(s)x}_X\,ds+\norm{A_1}\int_{-\tau}^{-\tau+t}\norm{f(s)}_X\,ds\\
  &\leq tM\norm{A_1}\norm{x}_{X}+\norm{A_1}\bigg(\int_{-\tau}^{-\tau+t}\norm{f(s)}_{X}^2ds\bigg)^{\frac{1}{2}}\bigg(\int_{-\tau}^{-\tau+t}1^2ds\bigg)^{\frac{1}{2}}\\
  &\leq tM\norm{A_1}\norm{x}_{X}+t^{\frac{1}{2}}\norm{A_1}\norm{f}_{L^2}\leq t^{\frac{1}{2}}M\norm{A_1}\big(\norm{x}_{X}+\norm{f}_{L^2}\big)\\
  &\leq (2t)^{\frac{1}{2}}M\norm{A_1}\norm{v}_\mathcal{X},
  \end{align*}
  where we used H{\"o}lder's inequality and the fact that 
\[
  (\norm{x}_X+\norm{f}_{L^2(-\tau,0;X)})^2 \leq 2(\norm{x}_{X}^2+\norm{f}_{L^2(-\tau,0;X)}^2),
\]
 with $\norm{v}_\mathcal{X}=(\norm{x}_{X}^2+\norm{f}_{L^2(-\tau,0;X)}^2)^{\frac{1}{2}}$.
Setting now $t_0$ small enough so that
  \[
  q:=(2t_0)^{\frac{1}{2}}M\norm{A_1}<1
  \]
we arrive at our conclusion.
\end{proof}
\begin{rem}
The operator $\Psi$ defined in \eqref{eq:delay_operator_def} is a special case of a much wider class of operators that satisfy \eqref{eq:condition_on_perturbation_for_well-posedness} and thus $(\calA,D(\calA))$ in \eqref{eq:abstract_A_def} remains a generator of a strongly continuous semigroup $(\mathcal{T}(t))_{t\geq0}$ on $\mathcal{X}$. For the proof of this general case see \cite[Section 3.3.3]{Batkai_Piazzera}. 
\end{rem}

We obtained results in Proposition~\ref{prop:adjoint_of_abstract_A} and Proposition~\ref{prop:non_homogeneous_abstract_Cauchy_problem_well-posedness} only by specifying a particular type of delay operator in the general setting of Section~\ref{sec:2}. Let us now specify the state space as $X:=l^2$ with the standard orthonormal basis $(e_k)_{k\in\NN}$, $(A,D(A))$ is a diagonal generator of a $C_0$-semigroup $(T(t))_{t\geq0}$ on $X$ with a sequence of eigenvalues $(\lambda_k)_{k\in\NN}\subset\CC$ such that 
\begin{equation}\label{eq:sequence_of_eigenvalues_of_the_generator}
\sup_{k\in\NN}\Re\lambda_k<\infty,
\end{equation}
and $A_1\in\calL(X)$ is a diagonal operator with a sequence of eigenvalues $(\gamma_k)_{k\in\NN}\subset\CC$. 
In other words, we introduce a finite-time state delay into the standard setting for diagonal systems \cite[Chapter 2.6]{Tucsnak_Weiss}. Hence, the $C_0$-semigroup generator $(A,D(A))$ is given by 
\begin{align}\label{eq:diagonal_generator_def}
D(A)=\bigg\{z\in l^2(\CC):\sum_{k\in\NN}(1+|\lambda_k|^2)|z_k|^2<\infty\bigg\},\quad (Az)_k=\lambda_{k}z_{k}.
\end{align}

Making use of the pivot duality, as the space $X_1$ we take  $(D(A),\norm{\cdot}_{gr})$, where the graph norm $\norm{\cdot}_{gr}$ is equivalent to
\[
\norm{z}_1^2=\sum_{k\in\NN}(1+|\lambda_k|^2)|z_k|^2.
\]
The adjoint generator $(A^*,D(A^*))$ has the form 
\begin{align}\label{eq:adjoint_of_diagonal_generator}
D(A^*)=D(A),\qquad (A^*z)_k=\overbar{\lambda}_{k}z_{k}.
\end{align}
The space $X_{-1}$ consists of all sequences $z=(z_k)_{k\in\NN}\subset\CC$ for which
\begin{align}\label{eq:diagonal_system_extrapolation_space_norm_equiv}
\sum_{k\in\NN}\frac{|z_k|^2}{1+|\lambda_k|^2}<\infty,
\end{align}
and the square root of the above series gives an equivalent norm on $X_{-1}$.  By Proposition~\ref{prop:Sobolev_tower_adjoint} the space $X_{-1}$ can be written as $(D(A^*))'$. Note also that the operator $B\in\mathcal{L}(\CC,X_{-1})$ is represented by the sequence $(b_k)_{k\in\NN}\subset\CC$ as
$\mathcal{L}(\CC,X_{-1})$ can be identified with $X_{-1}$. 

This completes the description of the setting for a diagonal retarded system. From now on we consider system \eqref{eq:retarded_non-autonomous_system} reformulated as \eqref{eq:delay_non-autonomous_diff_eq} and its Cauchy problem representation \eqref{eq:non-autonomous_abstract_Cauchy_problem_def} as defined with the diagonal elements described in this section.
\subsection{Analysis of a single component}\label{subsection:analysis_of_a_single_component}
Let us now focus on the $k$-th component of \eqref{eq:retarded_non-autonomous_system}, namely 
\begin{equation}\label{eq:k-th_component_of_diagonal_non-autonomous_system}
\left\{\begin{array}{ll}
        \dot{z}_k(t)=\lambda_{k}z_{k}(t)+\gamma_{k}z_{k}(t-\tau)+b_ku(t)\\
        z_k(0)=x_k,        \\
        z_{0_k}=f_k,		\\
       \end{array}
\right.
\end{equation}
where $\lambda_k,\gamma_k,b_k,x_k\in\CC$, $f_k:=\langle f,l_k\rangle_{L^2(-\tau,0;X)}l_k$ with $l_k$ being the $k$-th component of an orthonormal basis in $L^2(-\tau,0;X)$ (see \cite[Chapter 3.5, p.138]{Balakrishnan} for a description of such bases). Here $b_k$ is the $k$th component of $B$.

For clarity of notation, until the end of this subsection, we drop the subscript $k$ and rewrite \eqref{eq:k-th_component_of_diagonal_non-autonomous_system} in the form
\begin{equation}\label{eq:k-th_component_Cauchy_problem}
\left\{\begin{array}{ll}
        \dot{z}(t)=\lambda z(t)+\Psi z_t+bu(t)\\
        z(0)=x,        \\
        z_{0}=f,		\\
       \end{array}
\right.
\end{equation}
where  the delay operator $\Psi\in\mathcal{L}(W^{1,2}(-\tau,0;\CC),\CC)$ is given by
\begin{equation}\label{eq:k-th_component_delay_operator}
\Psi(f)=\gamma f(-\tau)\qquad\forall f\in W^{1,2}(-\tau,0;\CC).
\end{equation}
The setting for the $k$-th component now includes the extended state space 
\begin{equation}\label{eq:k-th_component_Cartesian_product_def}
\mathcal{X}:=\CC\times L^2(-\tau,0;\CC) 
\end{equation}
with an inner product 
\begin{equation}\label{eq:k-th_component_inner_product_Cartesian_product}
 \bigg\langle\binom{x}{f},\binom{y}{g}\bigg\rangle_{\mathcal{X}}:=x\bar{y}+\langle f,g\rangle_{L^2(-\tau,0;\CC)}\quad\forall \binom{x}{f},\binom{y}{g}\in\mathcal{X}.
\end{equation}
The Cauchy problem for the $k$-th component is
\begin{equation}\label{eq:k-th_component_abstract_Cauchy_problem}
\left\{\begin{array}{ll}
        \dot{v}(t)=\mathcal{A}v(t)+\mathcal{B}u(t)\\
        v(0)=\binom{x}{f},        \\
       \end{array}
\right.
\end{equation} 
where $v:[0,\infty)\ni t\mapsto\binom{z(t)}{z_t}\in\mathcal{X}$ and $\mathcal{A}$ is an operator on $D(\calA)\subset\mathcal{X}$ defined as
\begin{equation}\label{eq:k-th_component_abstract_A_domain_def}
D(\mathcal{A}):=\bigg\{\binom{x}{f}\in \CC\times W^{1,2}(-\tau,0;\CC):\ f(0)=x\bigg\},
\end{equation}
\begin{equation}\label{eq:k-th_component_abstract_A_def}
\mathcal{A}:=\left(\begin{array}{cc} 
								\lambda & \Psi \\ 
								0 & \frac{d}{d\sigma}	
						\end{array}\right)
\end{equation}
and $\mathcal{B}:=\binom{b}{0}\in\mathcal{L}(\CC,\mathcal{X}_{-1})$. By Proposition~\ref{prop:Sobolev_tower_adjoint} and Proposition~\ref{prop:adjoint_of_abstract_A} for the $k$-th component we have
\begin{equation}\label{eq:k-th_component_extended_space_def}
\mathcal{X}_{-1}=D(\calA^*)',
\end{equation}
where
\begin{equation}\label{eq:k-th_component_abstract_A*_domain_def}
D(\mathcal{A^*})=\bigg\{\binom{y}{g}\in \CC\times W^{1,2}(-\tau,0;\CC):\ \overbar{\gamma}\,y=g(-\tau)\bigg\},
\end{equation}
\begin{equation}\label{eq:k-th_component_abstract_A*_def}
\mathcal{A^*}\binom{y}{g}=\left(\begin{array}{c} 
								\overbar{\lambda}y+g(0) \\ 
								-\frac{d}{d\sigma}g	
						\end{array}\right),
\end{equation}
and $D(\calA^*)'$ is the dual to $D(\calA^*)$ with respect to the pivot space $\calX$ in \eqref{eq:k-th_component_Cartesian_product_def}. 
As the proof is essentialy the same, we only state a $k$-th component version of Proposition~\ref{prop:non_homogeneous_abstract_Cauchy_problem_well-posedness}, namely
\begin{prop}\label{prop:wellposedness_of_k-th_component_ACP}
The operator $(\mathcal{A},D(\mathcal{A}))$ given by \eqref{eq:k-th_component_abstract_A_domain_def}--\eqref{eq:k-th_component_abstract_A_def} generates a strongly continuous semigroup $\big(\mathcal{T}(t)\big)_{t\geq0}$ on $\calX$ given by \eqref{eq:k-th_component_Cartesian_product_def}.\qed
\end{prop}

Now that we know that the $k$-th component Cauchy problem \eqref{eq:k-th_component_abstract_Cauchy_problem} is well-posed we can formally write its  $\mathcal{X}_{-1}$-valued mild solution as
\begin{equation}\label{eq:k-th_component_abstract_Cauchy_problem_mild_solution}
 v(t)=\mathcal{T}(t)v(0)+\int_{0}^{t}\mathcal{T}(t-s)\mathcal{B}u(s) \, ds,
\end{equation}
where the control operator is $\mathcal{B}=\binom{b}{0}\in\mathcal{L}(\CC,\mathcal{X}_{-1})$ and $\mathcal{T}(t)\in\mathcal{L}(\mathcal{X}_{-1})$  is the extension of the $C_0$-semigroup generated by $(\mathcal{A},D(\mathcal{A}))$ in \eqref{eq:k-th_component_abstract_A_domain_def}--\eqref{eq:k-th_component_abstract_A_def}. 

The following, being a corollary from Proposition~\ref{prop abstract A resolvent operator}, gives the form of the $k$-th component resolvent $R(s,\mathcal{A})$. 
\begin{prop}\label{prop:k-th_component_abstract_A_resolvent_operator}
For $s\in\mathbb{C}$ and for all $1\leq p<\infty$ there is 
\begin{equation}\label{k-th component condition on resolvent sets}
s\in\rho(\mathcal{A})\text{ if and only if } s\in\rho(\lambda+\Psi_{s}).
\end{equation}
Moreover, for $s\in\rho(\mathcal{A})$ the resolvent $R(s,\mathcal{A})$ is given by

\begin{equation}\label{eq:k-th_component_resolvent_of_abstract_A}
	R(s,\mathcal{A})=\left(\begin{array}{ll} 
								R(s,\lambda+\Psi_{s}) & R(s,\lambda+\Psi_{s})\Psi R(s,A_0) \\ 
								\epsilon_{s}R(s,\lambda+\Psi_{s}) & (\epsilon_{s}R(s,\lambda+\Psi_{s})\Psi+I)R(s,A_0)
								\end{array}\right),
\end{equation}
where $R(s,\lambda+\Psi_s)\in\mathcal{L}(\CC)$,
\begin{equation}
 R(s,\lambda+\Psi_s)=\frac{1}{s-\lambda-\gamma\exp^{-s\tau}} \quad \forall s\in\underrightarrow{\CC}_{\abs{\lambda}+\abs{\gamma}}
\end{equation}
and $R(s,A_0)\in\mathcal{L}(L^2(-\tau,0;\CC))$,
 \begin{equation}
R(s,A_0)f(r)=\int_{r}^{0}\exp^{s(r-t)}f(t) \,dt
 \quad r\in[-\tau,0]\quad \forall s\in\underrightarrow{\CC}_{\abs{\lambda}+\abs{\gamma}}.
\end{equation}
\end{prop}
\begin{proof}
The proof runs along the lines of \cite[Proposition 3.3]{Partington_Zawiski_2019} with necessary adjustments for the forms of diagonal operators involved.
\end{proof}
By Proposition~\ref{prop:k-th_component_abstract_A_resolvent_operator} the resolvent component $R(s,\lambda+\Psi_s)$ is analytic in $\underrightarrow{\CC}_{\abs{\lambda}+\abs{\gamma}}$. To ensure analyticity of $R(s,\lambda+\Psi_s)$ in $\CC_{+}$, as required to apply $H(\CC_+)$-based approach, we introduce the following sets.

\begin{rem}
We take the principal argument of $\lambda$ to be $\Arg\lambda\in(-\pi,\pi]$.
\end{rem}
Let $\DD_r\subset\CC$ be an open disc centred at $0$ with radius $r>0$. We shall require the following subset of the complex plane, depending on $\tau>0$ and $a\in(-\infty,\frac{1}{\tau}]$ and shown in 
Fig.~\ref{fig:Lambda_tau_a_for_different_a}, namely: \index{Greek@\textbf{Greek symbols}!k@$\Lambda_{\tau,a}$}
\begin{itemize}
\item for $a<0$:
\begin{align}\label{eq:lambda_tau_a_neg}
\begin{split}
\Lambda_{\tau,a}:=&\bigg\{\eta\in\CC\setminus\DD_{|a|}:\re{\eta}+a<0,\,|\eta|<|\eta_\pi|,\\
&\abs{\Arg\eta}>\tau\sqrt{\abs{\eta}^2-a^2}+\arctan\Big(-\frac{1}{a}\sqrt{\abs{\eta}^2-a^2}\Big)\bigg\}\cup\DD_{|a|},
\end{split}
\end{align}
where $\eta_\pi$ is such that
\[
\sqrt{\abs{\eta_\pi}^2-a^2}\tau+\arctan\bigg(-\frac{1}{a}\sqrt{\abs{\eta_\pi}^2-a^2}\bigg)=\pi;
\]
\item for $a=0$:
\begin{align}\label{eq:lambda_tau_a_zero}
\Lambda_{\tau,a}:=&\bigg\{\eta\in\CC\setminus\{0\}:\re\eta<0,\, |\eta|<\frac{\pi}{2\tau},\, \abs{\Arg\eta}>\tau\abs{\eta}+\frac{\pi}{2}\bigg\};
\end{align}
\item for $0<a\leq\frac{1}{\tau}$
\begin{align}\label{eq:lambda_tau_a_pos}
\begin{split}
\Lambda_{\tau,a}:=&\bigg\{\eta\in\CC:\re{\eta}+a<0,\,|\eta|<|\eta_\pi|,\\
&\abs{\Arg\eta}>\tau\sqrt{\abs{\eta}^2-a^2}+\arctan\Big(-\frac{1}{a}\sqrt{\abs{\eta}^2-a^2}\Big)+\pi\bigg\},
\end{split}
\end{align}
where $\eta_\pi$ is such that $\abs{\eta_\pi}>a$ and
\[
\sqrt{\abs{\eta_\pi}^2-a^2}\tau+\arctan\bigg(-\frac{1}{a}\sqrt{\abs{\eta_\pi}^2-a^2}\bigg)=0.
\] 
\end{itemize}
The analyticity of $R(s,\lambda+\Psi_s)$ in $\CC_{+}$ follows now from the following \cite{Kapica_Zawiski_2022}
\begin{prop}\label{cor:stability_of_complex_polynomial}
Let $\tau>0$ and let $\lambda,\gamma,\eta\in\CC$ such that $\lambda=a+ib\in\underleftarrow{\CC}_{\frac{1}{\tau}}$. Then
\begin{itemize}
\item[(i)]  every solution of the equation $s-a -\eta\exp^{-s\tau}=0$ belongs to $\CC_-$ if and only if  $\eta\in\Lambda_{\tau,a}$;
\item[(ii)] every solution of  
\begin{equation}\label{eq:complex_polynomial_full_zeros}
s-\lambda-\gamma\exp^{-s\tau}=0
\end{equation}
and its version with conjugate coefficients
\begin{equation}\label{eq:complex_polynomial_full_conjugate_zeros}
s-\overbar{\lambda}-\overbar{\gamma}\exp^{-s\tau}=0
\end{equation}
belongs to $\CC_-$ if and only if $\gamma\exp^{-ib\tau}\in\Lambda_{\tau,a}$.
\end{itemize}
\end{prop}

\begin{figure}[htb]
\begin{center}
\includegraphics[scale=0.6]{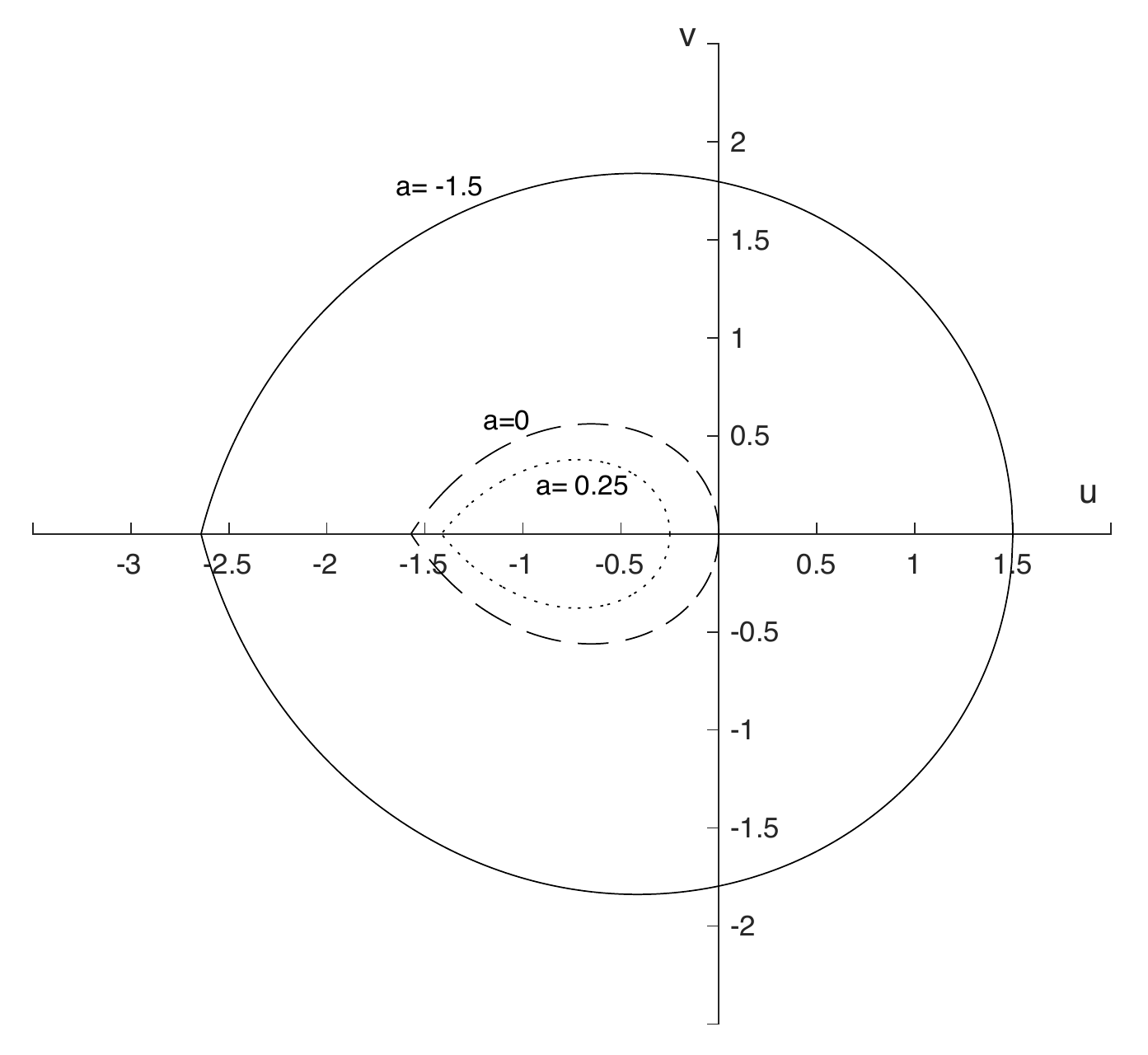}
\end{center}
\caption{Outer boundaries for some $\Lambda_{\tau,a}$ sets, defined in \eqref{eq:lambda_tau_a_neg}--\eqref{eq:lambda_tau_a_pos} with $\eta=u+iv$, for $\tau=1$ and different values of $a$: solid for $a=-1.5$, dashed for $a=0$ and dotted for $a=0.25$.}
\label{fig:Lambda_tau_a_for_different_a}
\end{figure}

In relation to the form of $R(s,\lambda+\Psi_s)$ consider the following technical result based on \cite{Walton_Marshall_1984}, originally stated for real coefficients, that for complex ones becomes
\begin{lem}\label{lem:cost_integral_in_TDS}
Let $\tau>0$ and $\lambda,\gamma\in\CC$ such that $\lambda=a+ib$ with $a\leq\frac{1}{\tau}$, $b\in\RR$ and $\gamma\exp^{-ib\tau}\in\Lambda_{\tau,a}$. Then 
\begin{equation}\label{eq:cost_integral_TDS_full_def}
J:=\frac{1}{2\pi}\int_{-\infty}^{\infty}\frac{d\omega}{\vert i\omega-\lambda-\gamma\exp^{-i\omega\tau}\vert^2}\\
=
\left\{\begin{array}{ll}
J_a, 	& |\gamma|<|a|\\
J_e,  &  |\gamma|=|a|\\
J_\gamma,  &   |\gamma|>|a|
\end{array}
\right.
\end{equation}

where
\begin{equation}\label{eq:cost_integral_TDS_Ja}
\begin{split}
J_a:=&\frac{1}{2\sqrt{a^2-|\gamma|^2}}\times\\
		&\times\frac{\exp^{\sqrt{a^2-|\gamma|^2}\tau}\big(a-\sqrt{a^2-|\gamma|^2}\big)
			+\exp^{-\sqrt{a^2-|\gamma|^2}\tau}\big(-a-\sqrt{a^2-|\gamma|^2}\big)}
			{2\Re(\gamma\exp^{-ib\tau})
			+\exp^{\sqrt{a^2-|\gamma|^2}\tau}\big(a-\sqrt{a^2-|\gamma|^2}\big)
			-\exp^{-\sqrt{a^2-|\gamma|^2}\tau}\big(-a-\sqrt{a^2-|\gamma|^2}\big)},
\end{split}
\end{equation}
\begin{equation}\label{eq:cost_integral_TDS_Je}
\begin{split}
J_e:=&\frac{1}{2}\frac{a\tau-1}{\re(\gamma\exp^{-ib\tau})+a},
\end{split}
\end{equation}
and
\begin{equation}\label{eq:cost_integral_TDS_Jg}
\begin{split}
J_\gamma:=&\frac{1}{2\sqrt{|\gamma|^2-a^2}}\times\\
	&\times\frac{a\sin(\sqrt{|\gamma|^2-a^2}\tau)-\sqrt{|\gamma|^2-a^2}\cos(\sqrt{|\gamma|^2-a^2}\tau)}
			{\Re(\gamma\exp^{-ib\tau})+a\cos(\sqrt{|\gamma|^2-a^2}\tau)+\sqrt{|\gamma|^2-a^2}\sin(\sqrt{|\gamma|^2-a^2}\tau)}.
\end{split}
\end{equation}
\end{lem}
The proof of Lemma~\ref{lem:cost_integral_in_TDS} is a rather technical one and so it is in the Appendix section. We easily obtain
\begin{cor}\label{cor:cost_integral_in_TDS_without_lambda}
Let $\tau>0$, $\lambda=0$ and $\gamma\in\Lambda_{\tau,0}$. Then 
\begin{equation}\label{eq:cost_integral_TDS_without_lambda}
J_0:=\frac{1}{2\pi}\int_{-\infty}^{\infty}\frac{d\omega}{\vert i\omega-\gamma\exp^{-i\omega\tau}\vert^2}
=-\frac{\cos(|\gamma|\tau)}{2\big(\Re(\gamma)+|\gamma|\sin(|\gamma|\tau)\big)}.
\end{equation}
\qed
\end{cor}

Referring to \eqref{eq:forcing_operator_infty_def} and the mild solution of the $k$-th component \eqref{eq:k-th_component_abstract_Cauchy_problem_mild_solution} the infinite-time forcing operator $\Phi_{\infty}\in\mathcal{L}(L^2(0,\infty;\CC),\mathcal{X}_{-1})$ is given by
\begin{equation}\label{eq:forcing_operator_infty_for_abstract_delay_system}
\Phi_{\infty}(u):=\int_{0}^{\infty}\mathcal{T}(t)\mathcal{B}u(t) \, dt,
\end{equation}
where
\[
\mathcal{T}(t)\mathcal{B}=\left(\begin{array}{cc} 
								\mathcal{T}_{11}(t) & \mathcal{T}_{12}(t) \\ 
								\mathcal{T}_{21}(t) & \mathcal{T}_{22}(t)	
						\end{array}\right)
						\left(\begin{array}{c} 
								b \\ 
								0	
						\end{array}\right)=
						\left(\begin{array}{c} 
								b\,\mathcal{T}_{11}(t)\\ 
								b\,\mathcal{T}_{21}(t)	
						\end{array}\right).
\]
Hence the forcing operator \eqref{eq:forcing_operator_infty_for_abstract_delay_system} becomes
\begin{equation}\label{eq:k-th_component_forcing_operator}
\Phi_{\infty}(u)=\left(\begin{array}{c} 
								\int_{0}^{\infty}\mathcal{T}_{11}(t)bu(t) \, dt\\ 
								\\
								\int_{0}^{\infty}\mathcal{T}_{21}(t)bu(t) \, dt	
						\end{array}\right)\in\mathcal{X}_{-1}=D(\calA^*)'.
\end{equation}
We can represent formally a similar product with the resolvent $R(s,\mathcal{A})$ from \eqref{eq:k-th_component_resolvent_of_abstract_A}, namely
\begin{equation}\label{eq:k-th_component_resolvent_times_B}
R(s,\mathcal{A})\mathcal{B}=\left(\begin{array}{cc} 
									R_{11}(s) & R_{12}(s) \\ 
									R_{21}(s) & R_{22}(s)	
						\end{array}\right)
						\left(\begin{array}{c} 
								b \\ 
								0	
						\end{array}\right)=
						\frac{b}{s-\lambda -\gamma\exp^{-s\tau}}
						\left(\begin{array}{c} 
								1\\ 
								\epsilon_s	
						\end{array}\right),					
\end{equation}
where the correspondence of sub-indices with elements of \eqref{eq:k-th_component_resolvent_of_abstract_A} is the obvious one and will be used from now on to shorten the notation. 

The connection between the $C_0$-semigroup $\mathcal{T}(t)$ and the resolvent $R(s,\mathcal{A})$ is given by the Laplace transform, whenever the integral converges, and 
\begin{equation}\label{eq:k-th_component_resolvent_as_Laplace_transform}
R(s,\mathcal{A})\mathcal{B}=\int_{0}^{\infty}\exp^{-sr}\mathcal{T}(r)\mathcal{B} \, dr=b\left(\begin{array}{c} 
								\mathcal{L}(\mathcal{T}_{11})(s)\\ 
								\\
								\mathcal{L}(\mathcal{T}_{21})(s)	
						\end{array}\right)\in\mathcal{L}(\CC,\mathcal{X}_{-1}).
\end{equation}

\begin{thm}\label{thm k-th component infty admissiblity}
Suppose that for a given delay $\tau>0$ there is  $\lambda=a+i\beta\in\underleftarrow{\CC}_{\frac{1}{\tau}}$ and  $\gamma\exp^{-i\beta\tau}\in \Lambda_{\tau,a}$. Then the control operator $\mathcal{B}=\binom{b}{0}$ for the system \eqref{eq:k-th_component_abstract_Cauchy_problem} is infinite-time admissible for every $u\in L^2(0,\infty;\CC)$ and 
\[
\norm{\Phi_{\infty}(u)}_{\mathcal{X}}^2\leq(1+\tau)\abs{b}^{2}J\norm{u}_{L^2(0,\infty;\CC)}^{2},
\]
where $J$ is given by \eqref{eq:cost_integral_TDS_full_def}.
\end{thm}
\begin{proof}
\begin{itemize}
\item[1.] Let the standard inner product on $L^2(0,\infty;\CC)$ be given by $\inner{f,g}_{L^2(0,\infty;\CC)}=\int_{0}^{\infty}f(t)\bar{g}(t) \, dt$ for every $f,g\in L^2(0,\infty;\CC)$.
Using \eqref{eq:k-th_component_forcing_operator} and \eqref{eq:k-th_component_extended_space_def} we may write for the first component of $\Phi_\infty(u)$
\begin{equation}\label{eq:forcing_operator_1st_component_as_inner_product}
\int_{0}^{\infty}\mathcal{T}_{11}(t)bu(t) \, dt=b\inner{\mathcal{T}_{11},\bar{u}}_{L^2(0,\infty;\CC)}
\end{equation}
assuming that $\mathcal{T}_{11}\in L^2(0,\infty;\CC)$. This assumption is equivalent, due to Theorem \ref{thm:Paley-Wiener}, to $\mathcal{L}(\mathcal{T}_{11})\in H^2(\CC_+)$, where the last inclusion holds. Indeed, using \eqref{eq:k-th_component_resolvent_times_B} and \eqref{eq:k-th_component_resolvent_as_Laplace_transform} we see that $\mathcal{L}(\mathcal{T}_{11})(s)=bR_{11}(s)=\frac{b}{s-\lambda-\gamma \exp^{-s\tau}}$. The assumptions on $\lambda$ and $\gamma$ give that $R_{11}$ is analytic in $\CC_+$. The boundary trace $R_{11}^*=\mathcal{L}(\mathcal{T}_{11})^*$ is given a.e. as
\[
\mathcal{L}(\mathcal{T}_{11})^*(i\omega)=\frac{1}{i\omega-\lambda -\gamma\exp^{-i\omega\tau}}.
\]
Lemma~\ref{lem:cost_integral_in_TDS} now gives that $\mathcal{L}(\mathcal{T}_{11})^*\in L^2(i\RR)$ and thus, by \eqref{eq:inner_product_on_H2(C+)}, $R_{11}\in H^2(\CC_+)$.
\item[2.] Again by Theorem \ref{thm:Paley-Wiener} and definition of the inner product on $H^2(\CC)_+$ in \eqref{eq:inner_product_on_H2(C+)} we have
\begin{align*}
b\inner{\mathcal{T}_{11},\bar{u}}_{L^2(0,\infty;\CC)}
=\frac{b}{2\pi}\int_{-\infty}^{+\infty}\frac{1}{i\omega-\lambda -\gamma\exp^{-i\omega\tau}}\overbar{\mathcal{L}(\bar{u})}^*(i\omega) \, d\omega.
\end{align*}
The Cauchy--Schwarz inequality now gives 
\begin{align*}
&\abs{b}\bigg\vert\frac{1}{2\pi}\int_{-\infty}^{+\infty}\frac{1}{i\omega-\lambda -\gamma \exp^{-i\omega\tau}}\overbar{\mathcal{L}(\bar{u})}^*(i\omega) \, d\omega\bigg\vert\\
&\leq\abs{b}\bigg(\frac{1}{2\pi}\int_{-\infty}^{+\infty} \Big\vert\frac{1}{i\omega-\lambda -\gamma\exp^{-i\omega\tau}}\Big\vert^{2} \, d\omega\bigg)^{\frac{1}{2}} \bigg(\frac{1}{2\pi}\int_{-\infty}^{+\infty} \big\vert \overbar{\mathcal{L}(\bar{u})}^*(i\omega)\big\vert^2 \, d\omega\bigg)^{\frac{1}{2}}\\
&=\abs{b}\, J^{\frac{1}{2}}\, \norm{u}_{L^2(0,\infty;\CC)},
\end{align*}
with $J$ given by \eqref{eq:cost_integral_TDS_full_def}. Combining this result with point 1 we obtain
\begin{equation}\label{eq:forcing_operator_1st_component_norm}
\bigg\vert\int_{0}^{\infty}\mathcal{T}_{11}(t)bu(t) \, dt\bigg\vert^2\leq\abs{b}^2 J \norm{u}_{L^2(0,\infty;\CC)}^2.
\end{equation}
\item[3.] Consider now the second element of the forcing operator \eqref{eq:k-th_component_forcing_operator}, namely
\[
\int_{0}^{\infty}\mathcal{T}_{21}(t)bu(t) \, dt\in W,
\]
where we denote by $W$ the second component of $\calX_{-1}=D(\calA^*)'$. If we assume that $\mathcal{T}_{21}\in  L^2(0,\infty;W)$ then using the vector-valued version of Theorem \ref{thm:Paley-Wiener} this is equivalent to $\mathcal{L}(\mathcal{T}_{21})\in H^2(\CC_+,W)$, but the last inclusion holds. Indeed, to show it notice that $R_{21}=\epsilon_s R_{11}$ where
\[
\epsilon_s(\sigma):=\exp^{s\sigma},\quad \sigma\in[-\tau,0],
\]
is, as a function of $s$, analytic everywhere for every value of  $\sigma$, and follow exactly the reasoning in point 1.
\item[4.] We introduce an auxiliary function $\phi:[0,\infty)\rightarrow\CC$. For that purpose fix $\mathcal{T}_{21}\in L^2(0,\infty;W)$ and $x_0\in W$ and define $\phi(t):=\inner{\mathcal{T}_{21}(t),x_0}_W$. Then $\phi\in L^2(0,\infty;\CC)$, as the Cauchy--Schwarz inequality gives 
\[
\int_{0}^{\infty}\abs{\inner{\mathcal{T}_{21}(t),x_0}_W}^{2} \, dt\leq\int_{0}^{\infty}\norm{\mathcal{T}_{21}(t)}_{W}^{2} \, dt\norm{x_0}_{W}^{2}<\infty.
\]

\item[5.] Consider now the following:
\[
b\int_{0}^{\infty}\phi(t)u(t) \, dt =b\int_{0}^{\infty}\inner{\mathcal{T}_{21}(t),x_0}_{W}u(t)dt 
 =b\bigg\langle\int_{0}^{\infty}\mathcal{T}_{21}(t)u(t) \, dt,x_0\bigg\rangle_{W}.
\]
We also have 
\begin{align*}
b\int_{0}^{\infty}\phi(t)u(t) \, dt=b\inner{\phi,\bar{u}}_{L^2(0,\infty;\CC)}=b\inner{\mathcal{L}(\phi)^*,\mathcal{L}(\bar{u})^*}_{L^2(i\RR)}.
\end{align*}
To obtain the boundary trace $\mathcal{L}(\phi)^*$ notice that
\begin{align*}
\mathcal{L}(\phi)(s)&=\int_{0}^{\infty}\exp^{-sr}\inner{\mathcal{T}_{21}(r),x_0}_{W} \, dr=\bigg\langle\int_{0}^{\infty}\exp^{-sr}\mathcal{T}_{21}(r) \, dr,x_0\bigg\rangle_{W}\\
&=\inner{\mathcal{L}(\mathcal{T}_{21})(s),x_0}_{W}=\inner{R_{21}(s),x_0}_{W}.
\end{align*}
Using now \eqref{eq:k-th_component_resolvent_times_B} yields the result
\[
\mathcal{L}(\phi)^*(i\omega)=\inner{R_{21}^*(i\omega),x_0}_{W}=\bigg\langle\frac{\epsilon_{i\omega}}{i\omega-\lambda-\gamma\exp^{-i\omega\tau}},x_0\bigg\rangle_{W}.
\]
Finally 
we obtain
\[
\bigg\langle\int_{0}^{\infty}\mathcal{T}_{21}(t)u(t) \, dt,x_0\bigg\rangle_{W}=\bigg\langle\frac{1}{2\pi}\int_{-\infty}^{+\infty}R_{21}^*(i\omega)\overbar{\mathcal{L}(\bar{u})}^*(i\omega) \, d\omega,x_0\bigg\rangle_{W}
\]
and 
\begin{equation}\label{eq:forcing_operator_2nd_component}
\int_{0}^{\infty}\mathcal{T}_{21}(t)u(t) \, dt
=\frac{1}{2\pi}\int_{-\infty}^{+\infty}R_{21}^*(i\omega)\overbar{\mathcal{L}(\bar{u})}^*(i\omega) \, d\omega\in W.
\end{equation}
\item [6.] By the definition of the norm on $L^2(-\tau,0;\CC)$ we have
\begin{align*}
\norm{R_{21}^*(i\omega)}_{L^2(-\tau,0;\CC)}^{2}&=\int_{-\tau}^{0}\bigg\vert \frac{\exp^{i\omega t}}{i\omega-\lambda -\gamma\exp^{-i\omega\tau}}\bigg\vert^{2} \, dt
= \frac{1}{\vert i\omega-\lambda -\gamma\exp^{-i\omega\tau}\vert^{2}}\int_{-\tau}^{0}\big\vert\exp^{i\omega t}\big\vert^{2} \, dt\\
&=\frac{\tau}{\vert i\omega-\lambda -\gamma\exp^{-i\omega\tau}\vert^{2}}.
\end{align*}

The Cauchy--Schwarz inequality gives
\begin{align*}
&\abs{b}\bigg\| \frac{1}{2\pi}\int_{-\infty}^{+\infty}R_{21}^*(i\omega)\overbar{\mathcal{L}(\bar{u})}^*(i\omega)
\, d\omega \bigg\|_{L^2(-\tau,0;\CC)}\\
&\leq\abs{b} \frac{1}{2\pi}\int_{-\infty}^{+\infty}\norm{R_{21}^*(i\omega)}_{L^2(-\tau,0;\CC)}\abs{\overbar{\mathcal{L}(\bar{u})}^*(i\omega)} \, d\omega\\
&=\abs{b} \frac{1}{2\pi}\int_{-\infty}^{+\infty}\frac{\tau^{\frac{1}{2}}}{\vert i\omega-\lambda -\gamma\exp^{-i\omega\tau}\vert}\abs{\overbar{\mathcal{L}(\bar{u})}^*(i\omega)} \, d\omega\\
&\leq\abs{b}\bigg(\frac{1}{2\pi}\int_{-\infty}^{+\infty} \Big(\frac{\tau^{\frac{1}{2}}}{\abs{i\omega-\lambda -\gamma\exp^{-i\omega\tau}}}\Big)^{2} \, d\omega\bigg)^{\frac{1}{2}} \bigg(\frac{1}{2\pi}\int_{-\infty}^{+\infty} \big\vert \overbar{\mathcal{L}(\bar{u})}^*(i\omega)\big\vert^2 \, d\omega\bigg)^{\frac{1}{2}}\\
&=\abs{b}\left(\tau J\right)^{\frac{1}{2}} \norm{u}_{L^2(0,\infty;\CC)},
\end{align*}
with $J$ given by \eqref{eq:cost_integral_TDS_full_def}. Combining this result with point 5 gives 
\begin{equation}\label{eq:forcing_operator_2nd_component_norm}
\bigg\|\int_{0}^{\infty}\mathcal{T}_{21}(t)bu(t) \, dt\bigg\|_{L^2(-\tau,0;\CC)}^2
\leq\abs{b}^2\tau J \norm{u}_{L^2(0,\infty;\CC)}^2.
\end{equation}

\item[7.] Taking now  the norm $\norm{\cdot}_{\mathcal{X}}$ resulting from \eqref{eq:k-th_component_inner_product_Cartesian_product} and using \eqref{eq:k-th_component_forcing_operator}, \eqref{eq:forcing_operator_1st_component_norm}, \eqref{eq:forcing_operator_2nd_component_norm} and Lemma \ref{lem:cost_integral_in_TDS} we arrive at
\begin{align}
\norm{\Phi_{\infty}(u)}_{\mathcal{X}}^{2}&=\bigg\vert\int_{0}^{\infty}\mathcal{T}_{11}(t)bu(t) \, dt\bigg\vert^2+\bigg\|\int_{0}^{\infty}\mathcal{T}_{21}(t)bu(t) \, dt\bigg\|_{L^2(-\tau,0;\CC)}^{2}\nonumber\\
&=(1+\tau)\abs{b}^{2}  J \norm{u}_{L^2(0,\infty;\CC)}^{2}.\label{eq:forcing_operator_bound}
\end{align}
\end{itemize}
\end{proof}
%
\subsection{Analysis of the whole retarded delay system}\label{subsection:Analysis of the whole retarded delay system}

Let us return to the diagonal system \eqref{eq:retarded_non-autonomous_system} reformulated as \eqref{eq:non-autonomous_abstract_Cauchy_problem_def} with the extended state space $\mathcal{X}=l^2\times L^2(-\tau,0;l^2)$ and the control operator  $\mathcal{B}\in\mathcal{L}(\CC,\mathcal{X}_{-1})$. We also return to denoting the $k$-th component of the extended state space with the subscript. By Proposition~\ref{prop:wellposedness_of_k-th_component_ACP} a mild solution of \eqref{eq:k-th_component_of_diagonal_non-autonomous_system}  is given by \eqref{eq:k-th_component_abstract_Cauchy_problem_mild_solution}, that is  $v_{k}:[0,\infty)\rightarrow\mathcal{X}$,
\begin{equation}\label{eq:k-th_component_abstract_Cauchy_problem_mild_solution_with_indices}
v_{k}(t)=\binom{z_{k}(t)}{z_{t_k}}=\mathcal{T}_{k}(t)v_{k}(0)+\int_{0}^{t}\mathcal{T}_{k}(t-s)\mathcal{B}_{k}u(s) \, ds.
\end{equation}
Given the structure of the Hilbert space $\mathcal{X}=l^2\times L^2(-\tau,0;l^2)$ in \eqref{eq:inner_product_on_Cartesian_product_def} the mild solution \eqref{eq:k-th_component_abstract_Cauchy_problem_mild_solution_with_indices} has values in the subspace of $\mathcal{X}$ spanned by the $k$-th element of its basis. Hence, defining $v:[0,\infty)\rightarrow\mathcal{X}$,
\begin{equation}\label{eq:abstract_Cauchy_problem_mild_solution_with_indices}
v(t):=\sum_{k\in\NN}v_{k}(t),
\end{equation}
we obtain the unique mild solution of \eqref{eq:non-autonomous_abstract_Cauchy_problem_def}. Using \eqref{eq:abstract_Cauchy_problem_mild_solution_with_indices} and \eqref{eq:inner_product_on_Cartesian_product_def} we have 
\begin{equation}\label{eq:abstract_Cauchy_problem_mild_solution_norm_in_X}
\begin{split}
\norm{v(t)}_{\mathcal{X}}^2&=\bigg\|\binom{z(t)}{z_{t}}\bigg\|_{\mathcal{X}}^2=\norm{z(t)}_{l^2}^2+\norm{z_t}_{L^2(-\tau,0;l^2)}^2\\
			   &=\sum_{k\in\NN}\abs{z_k(t)}^2+\sum_{k\in\NN}\abs{\inner{z_t,l_k}_{L^2(-\tau,0;l^2)}}^2\\
			   &=\sum_{k\in\NN}\bigg(\abs{z_k(t)}^2+\norm{z_{t_k}}_{L^2(-\tau,0;\CC)}^2\bigg)\\
			   &=\sum_{k\in\NN}\norm{v_k(t)}_{\mathcal{X}}^2,
\end{split}
\end{equation}
where we used again \eqref{eq:k-th_component_Cartesian_product_def} and notation from \eqref{eq:k-th_component_of_diagonal_non-autonomous_system}. We can formally write the mild solution \eqref{eq:abstract_Cauchy_problem_mild_solution_with_indices} as a function $v:[0,\infty)\rightarrow\mathcal{X}_{-1}$,

\begin{equation}\label{eq:whole_abstract_Cauchy_problem_mild_solution_formal}
v(t)=\mathcal{T}(t)v(0)+\int_{0}^{t}\mathcal{T}(t-s)\mathcal{B}u(s) \, ds.
\end{equation}
where the control operator $\mathcal{B}\in\mathcal{L}(\CC,\mathcal{X}_{-1})$ is given by  $\mathcal{B}=\binom{(b_k)_{k\in\NN}}{0}$.
We may now state the main theorem of this subsection.
\begin{thm}\label{thm:whole_system_infty_admissibility}
 Let for the given delay $\tau\in(0,\infty)$ sequences $(\lambda_k)_{k\in\NN}$ and $(\gamma_k)_{k\in\NN}$ be such that 
 \[
 \lambda_k=a_k+i\beta_k\in\underleftarrow{\CC}_{\frac{1}{\tau}}\quad\hbox{and}\quad  \gamma_k\exp^{-i\beta_k\tau}\in\Lambda_{\tau,a_k}\quad\forall k\in\NN,
 \]
with $\Lambda_{\tau,a_k}$ defined in \eqref{eq:lambda_tau_a_neg}-- \eqref{eq:lambda_tau_a_pos}. Then the control operator $\mathcal{B}\in\mathcal{L}(\CC,\mathcal{X}_{-1})$ given by $\mathcal{B}=\binom{(b_k)_{k\in\NN}}{0}$ is infinite-time admissible for system \eqref{eq:non-autonomous_abstract_Cauchy_problem_def} if the sequence $(C_k)_{k\in\NN}\in l^1$, where
\begin{equation}\label{eq:whole_sys_infty_admissibility_ceofficient}
C_k:=\abs{b_k}^2 J_{k}
\end{equation}
and $J_{k}$ is given by \eqref{eq:cost_integral_TDS_full_def} for every $(\lambda_k,\gamma_k)$, $k\in\NN$.
\end{thm}
\begin{proof}
Define the infinitie-time forcing operator for \eqref{eq:whole_abstract_Cauchy_problem_mild_solution_formal} as $\Phi_\infty:L^2(0,\infty)\rightarrow\mathcal{X}_{-1}$,
\[
 \Phi_{\infty}(u):=\int_{0}^{\infty}\mathcal{T}(t)\mathcal{B}u(t) \, dt.
\]
From \eqref{eq:abstract_Cauchy_problem_mild_solution_with_indices} it can be represented as
\begin{equation}\label{eq:forcing_operator_coordinate_decomposition}
 \Phi_{\infty}(u)=\sum_{k\in\NN}\Phi_{\infty_k}(u),
\end{equation}
where $\Phi_{\infty_k}(u)$ is given by 
\[
\Phi_{\infty_k}(u):=\int_{0}^{\infty}\mathcal{T}_{k}(t)\mathcal{B}_{k}u(t) \, dt,\quad k\in\NN.
\]
Then, similarly as in \eqref{eq:abstract_Cauchy_problem_mild_solution_norm_in_X} and using the assumption we see that 
\[
 \norm{\Phi_{\infty}(u)}_{\mathcal{X}}^2=\sum_{k\in\NN}\norm{\Phi_{\infty_k}(u)}_{\mathcal{X}}^2\leq(1+\tau)\bigg(\sum_{k\in\NN}\abs{C_k}\bigg)\norm{u}_{L^2(0,\infty;\CC)}^2<\infty.
\]
\end{proof}

Condition $(C_k)_{k\in\NN}\in l^1$ of Theorem~\ref{thm:whole_system_infty_admissibility} may not be easy to verify given the form of $J_k$ in \eqref{eq:cost_integral_TDS_full_def}. However, in certain situations the required condition follows from relatively simple relations between generator eigenvalues and a control sequence - see Example~\ref{example:unbounded_real_part_eigenvalues} below. 

The $l^1$-convergence condition was also used in \cite{Partington_Zawiski_2019}, where the results are, in fact, a special case of the present reasoning. This can be seen in Sections~\ref{subsec:Direct state-delayed diagonal systems} and \ref{example:bouneded_real_eigenvalues} below.

\section{Examples}\label{sec:4}

A motivating example of a dynamical system is the heat equation with delay \cite{Khusainov_Pokojovy_Azizbayov_2013}, \cite{Khodja_et_al_2014} (or a diffusion model with a delay in the reaction term \cite[Section 2.1]{Wu_1996}). Consider a homogeneous rod with zero temperature imposed on its both ends and its temperature change described by the following model
\begin{equation}\label{eq:heat_in_rod_by_PDEs}
\left\{\begin{array}{ll}
        \frac{\partial w}{\partial t}(x,t)=\frac{\partial^2 w}{\partial x^2}(x,t)+g(w(x,t-\tau)),& x\in(0,\pi),t\geq0,\\
        w(0,t)=0,\ w(\pi,t)=0,& t\in[0,\infty),\\
        w(x,0)=w_0(x), & x\in(0,\pi),\\
        w(x,t)=\varphi(x,t) & x\in(0,\pi),t\in[-\tau,0],
        \end{array}
\right.
\end{equation}
where the temperature profile $w(\cdot,t)$ belongs to the state space $X=L^2(0,\pi)$, initial condition is formed by the initial temperature distribution $w_0\in W^{2,2}(0,\pi)\cap W_{0}^{1,2}(0,\pi)$ and the initial history segment $\varphi_0\in W^{1,2}(-\tau,0;X)$, the action of $g$ is such that it can be considered as a linear and bounded diagonal operator on $X$. More precisely,  consider first \eqref{eq:heat_in_rod_by_PDEs} without the delay term i.e. the classical one-dimensional heat equation setting \cite[Chapter 2.6]{Tucsnak_Weiss}. Define
\begin{equation}\label{eq:heat_in_rod_generator_with_domain}
D(A):= W^{2,2}(0,\pi)\cap W_{0}^{1,2}(0,\pi),\quad Az:=\frac{d^2}{dx^2}z.
\end{equation}
Note that $0\in\rho(A)$. For $k\in\NN$ let $\phi_k\in D(A)$, $\phi_k(x):=\sqrt{\frac{2}{\pi}}\sin(kx)$ for every $x\in(0,\pi)$. Then $(\phi_k)_{k\in\NN}$ is an orthonormal Riesz basis in $X$ and 
\begin{equation}\label{eq:heat_in_rod_generator_eignevalues}
 A\phi_k=-k^2\phi_k \qquad \forall k\in\NN.
\end{equation}
Introduce now the delay term $g:X\to X$, $g(z):=A_1z$ where $A_1\in\calL(X)$ is such that $A_1\phi_k=\gamma_k\phi_k$ for every $k\in\NN$. We can now, using history segments, reformulate \eqref{eq:heat_in_rod_by_PDEs} into an abstract setting
\begin{equation}\label{eq:heat_in_rod_abstract}
\dot{z}(t)=Az(t)+A_1z_t(-\tau),\quad z(0)=w_0,\quad z_0=\varphi_0.
\end{equation}

Using standard Hilbert space methods and transforming system \eqref{eq:heat_in_rod_abstract} into the $l^2$ space (we use the same notation for the $l^2$ version of \eqref{eq:heat_in_rod_abstract}) and introducing control signal we obtain a retarded system of type \eqref{eq:retarded_non-autonomous_system}. The most important aspect of the above example is the sequence of eigenvalues $(\lambda_k)_{k\in\NN}=(-k^2)_{k\in\NN}$, a characteristic feature of the heat equation. Although the above heat equation is expressed using a specific Riesz basis, the idea behind remains the same. More precisely - one can redo the reasoning leading to a version of Theorem~\ref{thm:whole_system_infty_admissibility} based on a general Riesz basis instead of the standard orthonormal basis in $X$. Such approach, however, would be based on the same ideas and would inevitably suffer from a less clear presentation, and so we refrain from it.

\subsection{Eigenvalues with unbouded real part}\label{example:unbounded_real_part_eigenvalues}
Consider initially generators with unbounded real parts of their eigenvalues. For a given delay $\tau>0$ let a diagonal generator $(A,D(A))$ have a sequence of eigenvalues $(\lambda_k)_{k\in\NN}$ such that 
\begin{equation}\label{eq:example_unbounded_eigenvalues_lambda_assumption}
\lambda_k=a_k+i\beta_k\in\CC_-\quad \hbox{and}\quad a_k\to -\infty\ \hbox{as}\ k\to\infty.
\end{equation}
 
Let the operator $A_1\in\calL(X)$ be diagonal with a sequence of eigenvalues $(\gamma_k)_{k\in\NN}$. Boundedness of $A_1$ implies that there exists $M<\infty$ such that $|\gamma_k|\leq M$ for every $k\in\NN$. As $A_1$ is diagonal we easily get $|\gamma_k|\leq\|A_1\|\leq M$. Let the control operator $B$ be represented by the sequence $(b_k)_{k\in\NN}\subset\CC$.

To use Theorem~\ref{thm:whole_system_infty_admissibility} we need to assure additionally that $\gamma_k \exp^{i\beta_k\tau}\in\Lambda_{\tau,a_k}$ for every $k\in\NN$ and that the sequence $(C_k)_{k\in\NN}=(|b_k|^2 J_{a_k})_{k\in\NN}\in l^1$. However, for the former part we note that the boundedness of $A_1$ implies that there exists $N\in\NN$ such that
\begin{equation}\label{eq:example_unbounded_eigenvalues_N_index}
|\gamma_k|<|a_k|\quad\forall\ k>N.
\end{equation}
Fix such $N$. By the definition of $\Lambda_{\tau,a}$ in \eqref{eq:lambda_tau_a_neg} we see that $(\gamma_k)_{k>N}\subset\Lambda_{\tau,a_N}$. Thus the only additional assumption on operator $A_1$ we need is
\begin{equation}\label{eq:example_unbounded_eigenvalues_gamma_assumption}
\gamma_k\exp^{-i\beta_k\tau}\in\Lambda_{\tau,a_k}\quad\forall k\leq N.
\end{equation}
Assume that \eqref{eq:example_unbounded_eigenvalues_lambda_assumption} and \eqref{eq:example_unbounded_eigenvalues_gamma_assumption} hold. Then the sequence $(C_k)_{k\in\NN}\in l^1$ if and only if 
\[
\sum_{k\geq N}|C_k|=\sum_{k\geq N}|b_k|^2 J_{a_k}<\infty,
\] 
where $J_{a_k}$ is given by \eqref{eq:cost_integral_TDS_Ja} for every $k\geq N$. Let us denote $r_k:=\sqrt{a_k^2-|\gamma_k|^2}$. As $k\to\infty$ we have
\[
r_k\to\infty,\quad a_k-r_k\to -\infty,\quad a_k+r_k\to 0.
\]
and thus we obtain
\begin{equation*}
\begin{split}
\lim_{k\to\infty}\frac{|C_{k+1}|}{|C_k|}=&\lim_{k\to\infty}\frac{|b_{k+1}|^2}{|b_k|^2}\frac{J_{a_{k+1}}}{J_{a_k}}=\lim_{k\to\infty}\frac{|b_{k+1}|^2}{|b_k|^2}\frac{r_{k}}{r_{k+1}}\times\\
			&\times\frac{\exp^{r_{k+1}\tau}\big(a_{k+1}-r_{k+1}\big)
				+\exp^{-r_{k+1}\tau}\big(-a_{k+1}-r_{k+1}\big)}
				{2\Re(\gamma_{k+1}\exp^{-ib_{k+1}\tau})
				+\exp^{r_{k+1}\tau}\big(a_{k+1}-r_{k+1}\big)
				-\exp^{-r_{k+1}\tau}\big(-a_{k+1}-r_{k+1}\big)}\times\\
			&\times\frac{2\Re(\gamma_k\exp^{-ib_k\tau})
				+\exp^{r_k\tau}\big(a_k-r_k\big)
				-\exp^{-r_k\tau}\big(-a_k-r_k\big)}
				{\exp^{r_k\tau}\big(a_k-r_k\big)
				+\exp^{-r_k\tau}\big(-a_k-r_k\big)}\\
		=&\lim_{k\to\infty}\frac{|b_{k+1}|^2}{|b_k|^2}\frac{|a_k|\sqrt{1-\frac{|\gamma_k|^2}{a_k^2}}}{|a_{k+1}|\sqrt{1-\frac{|\gamma_{k+1}|^2}{a_{k+1}^2}}}\times\\
			&\times\frac{1-\exp^{-2r_{k+1}\tau}\frac{a_{k+1}+r_{k+1}}{a_{k+1}-r_{k+1}}}
				{1+\exp^{-r_{k+1}\tau}\frac{2\Re(\gamma_{k+1}\exp^{-ib_{k+1}\tau})}{a_{k+1}-r_{k+1}} 
				+\exp^{-2r_{k+1}\tau}\frac{a_{k+1}+r_{k+1}}{a_{k+1}-r_{k+1}}}\times\\
			&\times\frac{1+\exp^{-r_k\tau}\frac{2\Re(\gamma_k\exp^{-ib_k\tau})}{a_k-r_k} 
				+\exp^{-2r_k\tau}\frac{a_k+r_k}{a_k-r_k}}
				{1-\exp^{-2r_k\tau}\frac{a_k+r_k}{a_k-r_k}}\\
		=&\lim_{k\to\infty}\frac{|b_{k+1}|^2}{|b_k|^2}\frac{|a_k|}{|a_{k+1}|},
\end{split}
\end{equation*}
provided that at least one of these limits exists. The above results clearly depends on a particular set of eigenvalues.

Let us now look at the abstract heat equation \eqref{eq:heat_in_rod_abstract}. The sequence of eigenvalues in \eqref{eq:heat_in_rod_generator_eignevalues} i.e. $(\lambda_k)_{k\in\NN}=(-k^2)_{k\in\NN}$ clearly satisfies \eqref{eq:example_unbounded_eigenvalues_lambda_assumption}. For such $(\lambda_k)_{k\in\NN}$ we have
\begin{equation}\label{eq:example_unbounded_eigenvalues_C_k_ratio_limit}
\lim_{k\to\infty}\frac{\abs{C_{k+1}}}{\abs{C_k}}=\lim_{k\to\infty}\frac{\abs{b_{k+1}}^2}{\abs{b_k}^2},
\end{equation}
provided that at least one limit exists. By the d'Alembert series convergence criterion \\$\lim_{k\to\infty}\frac{\abs{C_{k+1}}}{\abs{C_k}}<1$ implies  $(C_k)_{k\in\NN}\in l^1$. 

Take the delay $\tau=1$ and assume that $A_1$ in \eqref{eq:heat_in_rod_abstract} is such that \eqref{eq:example_unbounded_eigenvalues_gamma_assumption} holds, i.e. there exists $N\in\NN$ such that $\gamma_k\in\Lambda_{1,-k^2}$ for every $k\leq N$ and $\gamma_k\in\Lambda_{1,-N^2}$ for every $k>N$.
Then, by Theorem~\ref{thm:whole_system_infty_admissibility} for $\calB=\binom{(b_k)_{k\in\NN}}{0}$ to be infinite-time admissible it is sufficient to take any $(b_k)_{k\in\NN}\in l^2$ such that 
\[
\lim_{k\to\infty}\frac{\abs{b_{k+1}}^2}{\abs{b_k}^2}<1.
\]
Note the role of the "first" eigenvalues $\gamma_k$ of $A_1$ which need to be inside consecutive $\Lambda_{\tau,-k^2}$ regions. As $A_1$ is a structural part of retarded system \eqref{eq:heat_in_rod_abstract} it may not always be possible to apply Theorem~\ref{thm:whole_system_infty_admissibility}.

\subsection{Direct state-delayed diagonal systems}\label{subsec:Direct state-delayed diagonal systems}
With small additional effort we can show that the so-called direct (or pure, see e.g. \cite{Baker_2000}) delayed system, where the delay is in the argument of the generator, is a special case of the problem analysed here.  Thus we apply our admissibility results to a dynamical system analysed in \cite{Partington_Zawiski_2019} and given by 
\begin{equation}\label{eq:diagonal_direct_delay_system}
\left\{\begin{array}{ll}
        \dot{z}(t)=Az(t-\tau)+Bu(t)\\
        z(0)=x,        \\
        z_0=f,		\\
       \end{array}
\right.
\end{equation}
where $(A,D(A))$ is a diagonal generator of a $C_0$-semigroup $(T(t))_{t\geq0}$ on $l^2$, $B$ is a control operator, $0<\tau<\infty$ is a delay and the control signal $u\in L^2(0,\infty;\CC)$. Let the sequence $(\lambda_k)_{k\in\NN}$ of the eigenvalues of $(A,D(A))$ be such that $\sup_{k\in\NN}\Re\lambda_k<0$.

We construct a setting as the one in Section~\ref{sec:3} and proceed with analysis of a $k-$th component, with a delay operator given again by point evaluation as $\Psi_k\in\mathcal{L}(W^{1,2}(-\tau,0;\CC),\CC)$, $\Psi_k(f):=\lambda_k f(-\tau)$ (we leave the index $k$ on purpose) and it is bounded as $\lambda_k$ is finite. The equivalent of \eqref{eq:k-th_component_of_diagonal_non-autonomous_system} now reads
\begin{equation}\label{eq:k-th_component_of_diagonal_non-autonomous_system_direct}
\left\{\begin{array}{ll}
        \dot{z}_k(t)=\lambda_{k}z_{k}(t-\tau)+b_ku(t)\\
        z_k(0)=x_k,        \\
        z_{0_k}=f_k,		\\
       \end{array}
\right.
\end{equation}
where the role of $\gamma_k$ in \eqref{eq:k-th_component_of_diagonal_non-autonomous_system} is played by $\lambda_k$ in \eqref{eq:k-th_component_of_diagonal_non-autonomous_system_direct}, while $\lambda_k$ of \eqref{eq:k-th_component_of_diagonal_non-autonomous_system} is $0$ in \eqref{eq:k-th_component_of_diagonal_non-autonomous_system_direct}, and this holds for every $k$. Thus, instead of a collection $\{\Lambda_{\tau,a_k}\}_{k\in\NN}$, we are concerned only with $\Lambda_{\tau,0}$. Using now Corollary~\ref{cor:cost_integral_in_TDS_without_lambda} instead of Lemma~\ref{lem:cost_integral_in_TDS}, the equivalent of Theorem~\ref{thm k-th component infty admissiblity} in the direct state-delayed setting takes the form
%
%
%
\begin{thm}\label{thm:k-th_component_infty_admissiblity_direct}
Let $\tau>0$ and take  $\lambda_k \in \Lambda_{\tau,0}$. Then the control operator $\mathcal{B}=\binom{b_k}{0}$ for the system based on \eqref{eq:k-th_component_of_diagonal_non-autonomous_system_direct}  is infinite-time admissible for every $u\in L^2(0,\infty;\CC)$ and 
\[
\norm{\Phi_{\infty}u}_{\mathcal{X}}^2\leq (1+\tau)\abs{b_k}^{2}\frac{-\cos(|\lambda_k|\tau)}{2\big(\Re(\lambda_k)+|\lambda_k|\sin(|\lambda_k|\tau)\big)}\norm{u}_{L^2(0,\infty;\CC)}^{2}.
\]
\qed
\end{thm}
As Theorem~\ref{thm:k-th_component_infty_admissiblity_direct} refers only to $k$-component it is an immediate consequence of Theorem~\ref{thm k-th component infty admissiblity}.  Using the same approach of summing over components the equivalent of Theorem~\ref{thm:whole_system_infty_admissibility} takes the form
\begin{thm}\label{thm:whole_system_infty_admissibility_direct}
 Let for the given delay $\tau\in(0,\infty)$ the sequence $(\lambda_k)_{k\in\NN}\subset\Lambda_{\tau,0}$. Then the control operator $\mathcal{B}\in\mathcal{L}(\CC,\mathcal{X}_{-1})$ for the system based on \eqref{eq:diagonal_direct_delay_system} and given by $\mathcal{B}=\binom{(b_k)_{k\in\NN}}{0}$ is infinite-time admissible if the sequence $(C_k)_{k\in\NN}\in l^1$, where
\begin{equation}\label{eq:whole_sys_infty_admissibility_ceofficient_direct}
 C_k:=\abs{b_k}^{2}\frac{-\cos(|\lambda_k|\tau)}{2\big(\Re(\lambda_k)+|\lambda_k|\sin(|\lambda_k|\tau)\big)}.
\end{equation}
\qed
\end{thm}

Note that the assumption that $\lambda_k\in\Lambda_{\tau,0}$ for every $k\in\NN$, due to boundedness of the $\Lambda_{\tau,0}$ set, implies that $A$ is in fact a bounded operator. While the result of Theorem~\ref{thm:whole_system_infty_admissibility_direct} is correct, it is not directly useful in analysis of unbounded operators. Instead, its usefulness follows from the the so-called reciprocal system approach. For a detailed presentation of the reciprocal system approach see \cite{Curtain_2003}, while for its application see \cite{Partington_Zawiski_2019}. We note here only that as there is some sort of symmetry in admissibility analysis of a given undelayed system and its reciprocal, introduction of a delay breaks this symmetry. In the current context consider the example of the next section.
\begin{rem}
In \cite{Partington_Zawiski_2019} the result corresponding to Theorem~\ref{thm:whole_system_infty_admissibility_direct} uses a sequence $(C_k)_{k\in\NN}$ which based not only on a control operator and eigenvalues of the generator, but also on some constants $\delta_k$ and $m_k$ so that $C_k=C_k(b_k,\lambda_k,\delta_k,m_k)$. As $\delta_k$ and $m_k$ originate from the proof of the result corresponding to Theorem~\ref{thm:k-th_component_infty_admissiblity_direct}, it requires additional effort to make the condition based on them useful. In the current form of Theorem~\ref{thm:whole_system_infty_admissibility_direct} this problem does not exist and the convergence of \eqref{eq:whole_sys_infty_admissibility_ceofficient_direct} depends only on the relation between eigenvalues of the generator and the control operator. 
\end{rem}
\subsection{Bounded real eigenvalues}\label{example:bouneded_real_eigenvalues}

In a diagonal framework of Example~\ref{subsec:Direct state-delayed diagonal systems} let us consider, for a given delay $\tau$, a sequence $(\lambda_k)_{k\in\NN}\subset\RR\cap\Lambda_{\tau,0}$ such that $\lambda_k\to 0$ as $k\to\infty$. In particular, let $\lambda_k:=(-\frac{\pi}{2}+\varepsilon)\tau^{-1}k^{-2}$ for some sufficiently small $0<\varepsilon<\frac{\pi}{2}$. Such sequence of $\lambda_k$ typically arises when considering a reciprocal system of a undelayed heat equation, as is easily seen by \eqref{eq:heat_in_rod_generator_eignevalues}. The ratio of absolute values of two consecutive coefficients \eqref{eq:whole_sys_infty_admissibility_ceofficient_direct} is
\[
\frac{\abs{C_{k+1}}}{\abs{C_k}}=\frac{\abs{b_{k+1}}^2}{\abs{b_k}^2}\frac{|\cos(\abs{\lambda_{k+1}}\tau)|}{|\cos(\abs{\lambda_k}\tau)|}\frac{\abs{\lambda_k}}{\abs{\lambda_{k+1}}}
\frac{\big|1-\sin(|\lambda_k|\tau)|\big|}{\big|1-\sin(|\lambda_{k+1}|\tau)|\big|}.
\]
It is easy to see that 
\begin{equation}\label{eq:dAlambert_limit}
\lim_{k\to\infty}\frac{\abs{C_{k+1}}}{\abs{C_k}}=\lim_{k\to\infty}\frac{\abs{b_{k+1}}^2}{\abs{b_k}^2},
\end{equation}
provided that at least one of these limits exists. By the d'Alembert series convergence criterion $\lim_{k\to\infty}\frac{\abs{C_{k+1}}}{\abs{C_k}}<1$ implies  $(C_k)_{k\in\NN}\in l^1$. Thus, by \eqref{eq:dAlambert_limit} for $\calB=\binom{(b_k)_{k\in\NN}}{0}$ to be infinite-time admissible for system \eqref{eq:diagonal_direct_delay_system} it is sufficient to take any $(b_k)_{k\in\NN}\in l^2$ such that 
\begin{equation}\label{eq:bounded_real_eig_example_b_condition}
\lim_{k\to\infty}\frac{\abs{b_{k+1}}^2}{\abs{b_k}^2}<1.
\end{equation}

\section{Appendix}
\subsection{Proof of Lemma \ref{lem:cost_integral_in_TDS}}
We rewrite $J$ as
\begin{equation}\label{eq:cost_integral_in_TDS}
\begin{split}
J=&\frac{1}{2\pi}\int_{-\infty}^{\infty}\frac{d\omega}{\vert i\omega-\lambda-\gamma\exp^{-i\omega\tau}\vert^2}\\
=&\frac{1}{2\pi }\int_{-\infty}^{\infty}\frac{d\omega}{(i\omega-\lambda-\gamma\exp^{-i\omega\tau})(-i\omega-\overbar{\lambda}-\overbar{\gamma}\exp^{i\omega\tau})}\\
=&\frac{1}{2\pi i}\int_{-i\infty}^{i\infty}\frac{ds}{(s-\lambda-\gamma\exp^{-s\tau})(-s-\overbar{\lambda}-\overbar{\gamma}\exp^{s\tau})}\\
=&\frac{1}{2\pi i}\int_{-i\infty}^{i\infty}E_1(s)E_2(s)\,ds,
\end{split}
\end{equation}
where
\begin{equation}\label{eq:cost_integra_integrands}
E_1(s):=\frac{1}{s-\lambda-\gamma\exp^{-s\tau}}, \qquad E_2(s):=\frac{1}{-s-\overbar{\lambda}-\overbar{\gamma}\exp^{s\tau}}.
\end{equation}
Note that writing explicitly $E_1$ and $E_2$ as functions of $s$ and parameters $\lambda$, $\gamma$ and $\tau$ we have
\[
E_1(s,\lambda,\gamma,\tau)=E_2(-s,\overbar{\lambda},\overbar{\gamma},\tau).
\]
Let $\mathcal{E}_1$ be the set of poles of $E_1$ and $\mathcal{E}_2$ be the set of poles of $E_2$. As, by assumption,  $\gamma\exp^{-ib\tau}\in\Lambda_{\tau,a}$ Proposition \ref{cor:stability_of_complex_polynomial} states that $\mathcal{E}_1\subset\CC_-$ and $\mathcal{E}_2\subset\CC_+$. Thus we have that $E_1$ is analytic in $\CC\setminus\CC_-$ while $E_2$ is analytic in $\CC\setminus\CC_+$. 

Let $s_n\in\mathcal{E}_1$, i.e. $s_n-\lambda-\gamma\exp^{-s_n\tau}=0$. Rearranging gives
\[
\frac{\gamma}{s_n-\lambda}=\exp^{s_n\tau}.
\] 
Substituting above to $E_2$ gives
\begin{equation}\label{eq:cost_integral_integrand_E2_at_poles_of_E1}
E_2(s_n)=-\frac{s_n-\lambda}{(s_n+\overbar{\lambda})(s_n-\lambda)+\abs{\gamma}^2}.
\end{equation}
and this value is finite as $s_n\not\in\mathcal{E}_2$.
Rearranging \eqref{eq:cost_integral_in_TDS} to account for  \eqref{eq:cost_integral_integrand_E2_at_poles_of_E1} gives
\begin{equation}\label{eq:cost_integral_extended}
J=\frac{1}{2\pi i}\int_{-i\infty}^{i\infty}\Bigg(E_1(s)\bigg(E_2(s)+\frac{s-\lambda}{(s+\overbar{\lambda})(s-\lambda)+\abs{\gamma}^2}\bigg)-E_1(s)\frac{s-\lambda}{(s+\overbar{\lambda})(s-\lambda)+\abs{\gamma}^2}\Bigg)\,ds.
\end{equation}
The above integrand has no poles at the roots $\{z_1,z_2\}$ of 
\begin{equation}\label{eq:cost_integral_additional_poles}
(s+\overbar{\lambda})(s-\lambda)+\abs{\gamma}^2=0.
\end{equation} 
However, as the two parts of the integrand in \eqref{eq:cost_integral_extended} will be treated separately, we need to consider poles introduced by $z_1$ and $z_2$ with regard to the contour of integration. Rewrite also \eqref{eq:cost_integral_additional_poles} as  
\begin{equation}\label{eq:cost_integral_additional_poles_definition}
(s+\overbar{\lambda})(s-\lambda)+\abs{\gamma}^2=(s-z_1)(s-z_2)=0.
\end{equation}
From this point onwards we analyse three cases given by the right side of \eqref{eq:cost_integral_TDS_full_def}. Assume first that $|\gamma|<|a|$. Then 
\begin{equation}\label{eq:cost_integral_additional_poles_Ja}
z_1=-\sqrt{a^2-|\gamma|^2}+ib,\quad z_2=\sqrt{a^2-|\gamma|^2}+ib.
\end{equation}
Figure~\ref{fig:cost_integral_contours}a shows integration contours $\Gamma_1(r)=\Gamma_I(r)+\Gamma_L(r)$ and $\Gamma_2(r)=\Gamma_I(r)+\Gamma_R(r)$ for $r\in(0,\infty)$ used for calculation of $J$. In particular  $\Gamma_I$ runs along the imaginary axis, $\Gamma_L$ is a left semicircle and $\Gamma_R$ is a right semicircle.
\begin{figure}[tb]\label{fig:Lambda_tau_blob}
\begin{center}
\includegraphics[scale=0.5]{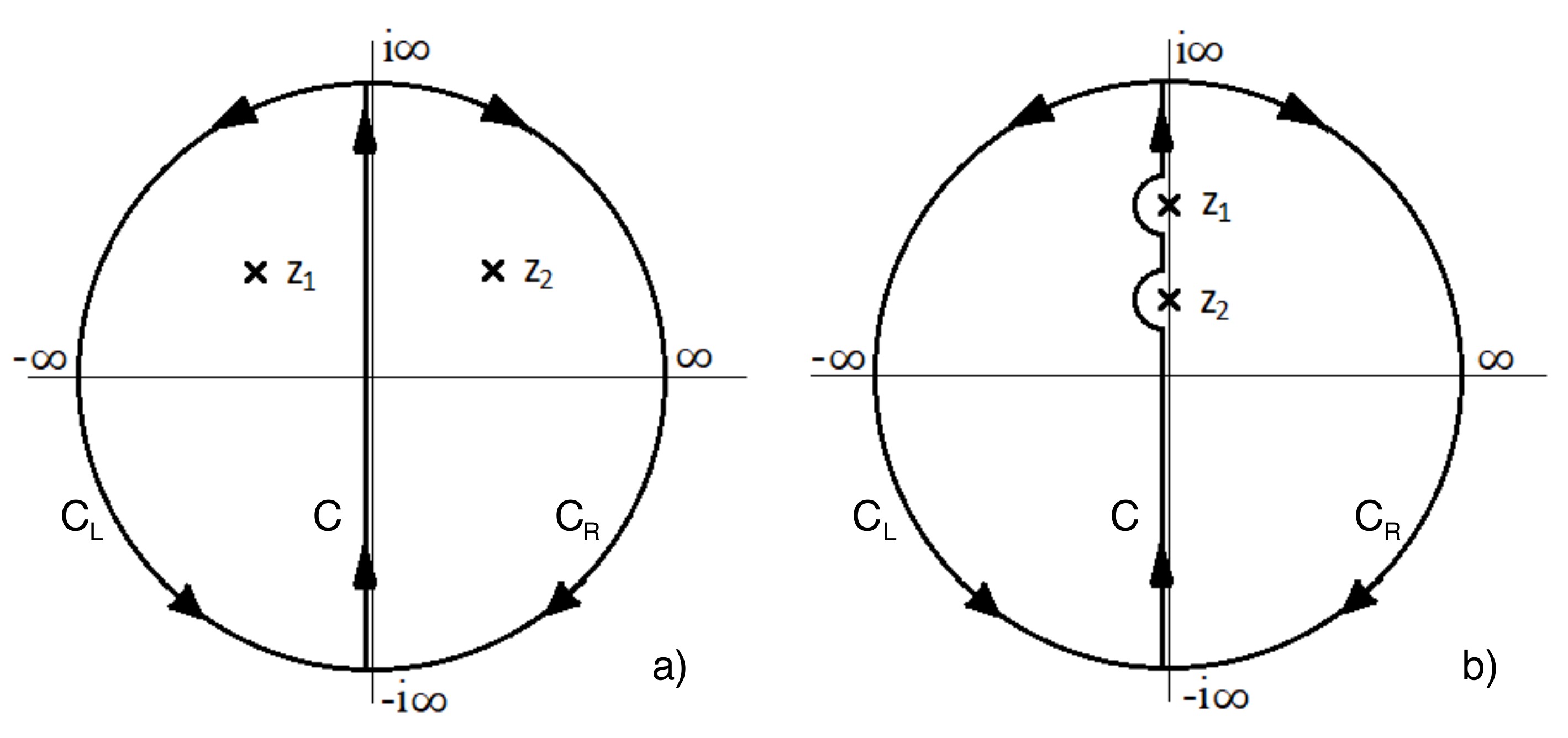} 
\end{center}
\caption{Contours of integration in \eqref{eq:cost_integral_extended}: part $a)$ is used for the case $|\gamma|<|a|$, part $b)$ is used when $|\gamma|\geq|a|$. Both parts are drawn for a sufficiently large $r$ so that $\Gamma_I(r)=C$, $\Gamma_L(r)=C_L$ and $\Gamma_R(r)=C_R$ and they enclose particular values of $z_1$ and $z_2$ in \eqref{eq:cost_integral_additional_poles_Ja} and \eqref{eq:cost_integral_additional_poles_Jg}, respectively. The location of infinitesimally small semicircles around $z_1$ and $z_2$ in part $b)$ is to be modified depending on the location of $z_1$ and $z_2$ on the imaginary axis.}
\label{fig:cost_integral_contours}
\end{figure}
Due to the above argument for a sufficiently large $r$ we get
\begin{equation}\label{eq:cost_integral_extended_contours}
\begin{split}
J&=\frac{1}{2\pi i}\lim_{r\to\infty}\int_{\Gamma_I(r)}\Bigg(E_1(s)\bigg(E_2(s)+\frac{s-\lambda}{(s+\overbar{\lambda})(s-\lambda)+\abs{\gamma}^2}\bigg)-E_1(s)\frac{s-\lambda}{(s+\overbar{\lambda})(s-\lambda)+\abs{\gamma}^2}\Bigg)\,ds\\
&=\frac{1}{2\pi i}\int_{C+C_L}E_1(s)\bigg(E_2(s)+\frac{s-\lambda}{(s-z_1)(s-z_2)}\bigg)\,ds\\
&\quad-\frac{1}{2\pi i}\int_{C+C_R}E_1(s)\frac{s-\lambda}{(s-z_1)(s-z_2)}\,ds.
\end{split}
\end{equation}
In calculation of the above we used the fact both integrals round the semicircles at infinity are zero as the integrands are, at most, of order $s^{-2}$ and for every fixed $\varphi, \lambda, \gamma, \tau$,
\begin{equation}\label{eq:cost_integral_at_infinity}
\lim_{r\rightarrow\infty}\frac{1}{r\exp^{i\varphi}-\lambda-\gamma\exp^{-r\exp^{i\varphi}\tau}}=0.
\end{equation}
Define separate parts of \eqref{eq:cost_integral_extended_contours} as 
\begin{equation}\label{eq:cost_integral_left_part}
J_L:=\frac{1}{2\pi i}\int_{C+C_L}E_1(s)\bigg(E_2(s)+\frac{s-\lambda}{(s-z_1)(s-z_2)}\bigg)\,ds
\end{equation}
and
\begin{equation}\label{eq:cost_integral_right_part}
J_R:=-\frac{1}{2\pi i}\int_{C+C_R}E_1(s)\frac{s-\lambda}{(s-z_1)(s-z_2)}\,ds
\end{equation}
and consider them separately. 

To calculate $J_L$ note that from \eqref{eq:cost_integral_integrand_E2_at_poles_of_E1} it follows that for every $s_n\in\mathcal{E}_1$ the value
\[
E_2(s_n)=-\frac{s_n-\lambda}{(s_n-z_1)(s_n-z_2)}
\]
is finite and that implies that $\{z_1,z_2\}\cap\mathcal{E}_1=\emptyset$. Thus the only pole of the integrand in \eqref{eq:cost_integral_left_part} encircled by the $C+C_L$ contour is at $z_1$. Denoting this integrand by $f$ the residue formula gives
\[
\hbox{Res}_{z_1}f(s)=\lim_{s\to z_1}(s-z_1)f(s)=E_1(z_1)\frac{z_1-\lambda}{z_1-z_2}.
\]
As the $C+C_L$ contour is counter-clockwise we obtain
\begin{equation}\label{eq:cost_integral_J_L_for_Ja}
J_L=E_1(z_1)\frac{z_1-\lambda}{z_1-z_2}.
\end{equation}
To calculate $J_R$ note that the only pole encircled by the $C+C_R$ contour is at $z_2$. Denoting the integrand of \eqref{eq:cost_integral_right_part} by $g$ the residue formula gives
\[
\hbox{Res}_{z_2}g(s)=\lim_{s\to z_2}(s-z_2)g(s)=E_1(z_2)\frac{z_2-\lambda}{z_2-z_1}.
\]
As the $C+C_R$ contour is clockwise we obtain
\begin{equation}\label{eq:cost_integral_J_R_for_Ja}
J_R=E_1(z_2)\frac{z_2-\lambda}{z_2-z_1}.
\end{equation}
Thus we obtain
\begin{equation}\label{eq:cost_integral_J_E_first_result}
J=J_L+J_R=\frac{1}{z_2-z_1}\Big((\lambda-z_1)E_1(z_1)+(z_2-\lambda)E_1(z_2)\Big),
\end{equation}
where $z_1,z_2$ are given by \eqref{eq:cost_integral_additional_poles_Ja}. We substitute these values for $z_1$ and $z_2$ and perform tedious calculations to obtain
\begin{align*}
J=&\frac{1}{2\sqrt{a^2-|\gamma|^2}}\times\\
	&\frac{\gamma\exp^{-ib\tau}\Big(\exp^{\sqrt{a^2-|\gamma|^2}\tau}(a-\sqrt{a^2-|\gamma|^2})+\exp^{-			\sqrt{a^2-|\gamma|^2}\tau}(-a-\sqrt{a^2-|\gamma|^2})\Big)}
		{\gamma\exp^{-ib\tau}\Big(\overbar{\gamma}\exp^{ib\tau}-\exp^{-\sqrt{a^2-|\gamma|^2}\tau}(-\sqrt{a^2-|\gamma|^2}-a)-\exp^{\sqrt{a^2-|\gamma|^2}\tau}(\sqrt{a^2-|\gamma|^2}-a)+\gamma\exp^{-ib\tau}\Big)}\\
		=&\frac{1}{2\sqrt{a^2-|\gamma|^2}}\times\\
	&\frac{\exp^{\sqrt{a^2-|\gamma|^2}\tau}\big(a-\sqrt{a^2-|\gamma|^2}\big)
			+\exp^{-\sqrt{a^2-|\gamma|^2}\tau}\big(-a-\sqrt{a^2-|\gamma|^2}\big)}
			{2\Re(\gamma\exp^{-ib\tau})
			+\exp^{\sqrt{a^2-|\gamma|^2}\tau}\big(a-\sqrt{a^2-|\gamma|^2}\big)
			-\exp^{-\sqrt{a^2-|\gamma|^2}\tau}\big(-a-\sqrt{a^2-|\gamma|^2}\big)}.
\end{align*}
Assume now that $|\gamma|>|a|$. The roots $\{z_1,z_2\}$ of \eqref{eq:cost_integral_additional_poles_definition} are
\begin{equation}\label{eq:cost_integral_additional_poles_Jg}
z_1=i\sqrt{a^2-|\gamma|^2}+ib,\quad z_2=-i\sqrt{a^2-|\gamma|^2}+ib.
\end{equation}
To calculate $J$ in \eqref{eq:cost_integral_extended} we now use the contour shown in Figure~\ref{fig:cost_integral_contours}b. We again define $J_L$ and $J_R$ as in \eqref{eq:cost_integral_left_part} and \eqref{eq:cost_integral_right_part}, respectively, but with this new contour. 

As $\gamma\exp^{-ib\tau}\in\Lambda_{\tau,a}$ by Proposition~\ref{cor:stability_of_complex_polynomial} no pole of $E_1$ lies on the imaginary axis. Hence no pole of the integrand in \eqref{eq:cost_integral_left_part} is encircled by the $C+C_L$ contour and this gives 
\begin{equation}\label{eq:cost_integral_J_L_for_Jg}
J_L=0.
\end{equation}
For $J_R$ the only poles of the integrand of \eqref{eq:cost_integral_right_part} encircled by the $C+C_R$ contour are $z_1$ and $z_2$. Denoting this integrand by $g$ the residue formula gives 
\[
\hbox{Res}_{z_1}g(s)=E_1(z_1)\frac{z_1-\lambda}{z_1-z_2},\quad \hbox{Res}_{z_2}g(s)=E_1(z_2)\frac{z_2-\lambda}{z_2-z_1}.
\]
As the $C+C_R$ contour is clockwise we obtain
\begin{equation}\label{eq:cost_integral_J_R_for_Jg}
J_R=E_1(z_1)\frac{z_1-\lambda}{z_1-z_2}+E_1(z_2)\frac{z_2-\lambda}{z_2-z_1}.
\end{equation}
Thus we obtain
\begin{equation}\label{eq:cost_integral_J_E_for_Jg}
J=J_L+J_R=\frac{1}{z_2-z_1}\Big((\lambda-z_1)E_1(z_1)+(z_2-\lambda)E_1(z_2)\Big),
\end{equation}
where $z_1,z_2$ are given by \eqref{eq:cost_integral_additional_poles_Jg}. Substituting these values, again after tedious calculations, we obtain 
\begin{align*}
J=&\frac{1}{2i\sqrt{|\gamma|^2-a^2}}\times\\
	&\frac{a\Big(\exp^{i\sqrt{|\gamma|^2-a^2}\tau}-\exp^{-i\sqrt{|\gamma|^2-a^2}\tau}\Big)
				-i\sqrt{|\gamma|^2-a^2}\Big(\exp^{i\sqrt{|\gamma|^2-a^2}\tau}+\exp^{-i\sqrt{|\gamma|^2-a^2}\tau}\Big)}
		{2\Re(\gamma\exp^{-ib\tau})+a\Big(\exp^{i\sqrt{|\gamma|^2-a^2}\tau}+\exp^{-i\sqrt{|\gamma|^2-a^2}\tau}\Big)
		-i\sqrt{|\gamma|^2-a^2}\Big(\exp^{i\sqrt{|\gamma|^2-a^2}\tau}-\exp^{-i\sqrt{|\gamma|^2-a^2}\tau}\Big)}\\
	=&\frac{1}{2\sqrt{|\gamma|^2-a^2}}\times\\
	&\frac{a\sin\big(\sqrt{|\gamma|^2-a^2}\tau\big)-\sqrt{|\gamma|^2-a^2}\cos\big(\sqrt{|\gamma|^2-a^2}\tau\big)}{\Re(\gamma\exp^{-ib\tau})+a\cos\big(\sqrt{|\gamma|^2-a^2}\tau\big)+\sqrt{|\gamma|^2-a^2}\sin\big(\sqrt{|\gamma|^2-a^2}\tau\big)}.
\end{align*}
For the last case assume that $|\gamma|=|a|>0$, as the assumption $\gamma\exp^{-b\tau}\in\Lambda_{\tau,0}$ excludes the case $|a|=|\gamma|=0$ because $0\not\in\Lambda_{\tau,0}$. Instead of $\{z_1,z_2\}$ we now have a single double root $z_0$ of \eqref{eq:cost_integral_additional_poles_definition} given by
\begin{equation}\label{eq:cost_integral_additional_poles_Je}
z_0=ib.
\end{equation}
As $z_0$ lies on the imaginary axis we use the contour shown in Figure~\ref{fig:cost_integral_contours}b tailored to the case $z_1=z_2=z_0$. Define $J_L$ and $J_R$ as in \eqref{eq:cost_integral_left_part} and \eqref{eq:cost_integral_right_part}, respectively, but with the contour tailored for $z_0$. For the same reasons as in \eqref{eq:cost_integral_J_L_for_Jg} we have 
\begin{equation}\label{eq:cost_integral_J_L_for_Je}
J_L=0.
\end{equation}
For $J_R$ the only pole of the integrand of \eqref{eq:cost_integral_right_part} encircled by the $C+C_R$ contour is $z_0$. Denoting this integrand by $g$ the residue formula for a double root gives 
\[
\hbox{Res}_{z_0}g(s)=\lim_{s\to z_0}\frac{d}{ds}\left((s-z_0)^2g(s)\right)=\frac{(a\tau-1)\gamma\exp^{-ib\tau}}{(a+\gamma\exp^{-ib\tau})^2}.
\]
With the current assumption we have that $a^2=\gamma\overbar{\gamma}$. By this and the fact that the $C+C_R$ contour is clockwise we obtain
\begin{equation}\label{eq:cost_integral_J_R_for_Je}
J_R=\frac{1}{2}\frac{a\tau-1}{\re(\gamma\exp^{-ib\tau})+a}.
\end{equation}
As $J=J_L+J_R$ this finishes the proof.\qed
\section{Acknowledgements}
The authors would like to thank Prof. Yuriy Tomilov for many valuable comments and mentioning to them reference \cite{Kappel_1986}. 

\section{Declarations as required by Springer Nature}
\subsection{Funding}

Rafał Kapica's research was supported by the Faculty of Applied Mathematics AGH UST statutory tasks within subsidy of Ministry of Education and Science.\\

\noindent Jonathan R. Partington indicates no external funding.\\

Radosław Zawiski's work was performed when he was a visiting researcher at the Centre for Mathematical Sciences of the Lund University, hosted by Sandra Pott, and supported by the Polish National Agency for Academic Exchange (NAWA) within the Bekker programme under the agreement PPN/BEK/2020/1/00226/U/00001/A/00001.

\subsection{Competing interests}
The authors have no competing interests as defined by Springer, or other interests that might be perceived to influence the results and/or discussion reported in this paper.

\subsection{Authors' contributions}
JRP and RZ are responsible for the initial conception of the research and the approach method.  
RZ performed the research concerning every element needed for the single component as well as the whole system analysis.  Examples for unbounded generators were provided by RK and RZ, while examples for direct state-delayed systems come from the work of JRP and RZ. Figures 1 - 2 were prepared by RZ. All authors participated in writing the manuscript. All authors reviewed the manuscript.


\begin{thebibliography}{10}

\bibitem{Amann}
H.~Amann, \emph{{Linear and Quasilinear Parabolic Problems}}, Monographs in
  Mathematics, vol.~89, Birkh{\"a}user Basel, Basel, 1995.

\bibitem{Arendt_et_al}
W.~{Arendt}, C.J.K {Batty}, M.~{Hieber}, and F.~{Neubrander},
  \emph{{Vector-valued Laplace Transforms and Cauchy Problems}}, 2nd ed.,
  Monographs in Mathematics, vol.~96, Birkh{\"a}user Verlag AG, Basel, 2010.

\bibitem{Baker_2000}
C.~T.~H. {Baker}, \emph{Retarded differential equations}, Journal of
  Computational and Applied Mathematics \textbf{125} (2000), 309--335.

\bibitem{Balakrishnan}
A.~V. Balakrishnan, \emph{Applied functional analysis}, second ed., Springer,
  New York, 1981.

\bibitem{Batkai_Piazzera}
A.~{Batk{\'a}i} and S.~{Piazzera}, \emph{{Semigroups for Delay Equations}},
  Research Notes in Mathematics, vol.~10, CRC Press, 2005.

\bibitem{Brezis}
H.~Brezis, \emph{{Functional Analysis, Sobolev Spaces and Partial Differential
  Equations}}, Universitext, Springer-Verlag New York, New York, 2011.
  
\bibitem{Curtain_2003}
R. F. Curtain, \emph{Regular linear systems and their reciprocals: applications to {Riccati} equations}, Systems and Control Letters \textbf{49} (2003), 81--89.

\bibitem{Engel_1999}
K.-J. Engel, \emph{Spectral theory and generator property for one-sided coupled
  operator matrices}, Semigroup Forum \textbf{58(2)} (1999), 267--295.

\bibitem{Engel_Nagel}
K.-J. {Engel} and R.~{Nagel}, \emph{{One-Parameter Semigroup for Linear
  Evolution Equations}}, Graduate Texts in Mathematics, vol. 194,
  Springer-Verlag, Berlin, 2000.

\bibitem{Evans}
L.~C. {Evans}, \emph{{Partial Differential Equations}}, {Graduate Studies in
  Mathematics}, vol.~19, American Mathematical Society, 2002.

\bibitem{Garnett}
J.~{Garnett}, \emph{{Bounded Analytic Functions}}, Graduate Texts in
  Mathematics, vol. 236, Springer-Verlag New York, Basel, 2007.

\bibitem{Grabowski_Callier_1996}
P.~Grabowski and F.~M. Callier, \emph{Admissible observation operators,
  semigroup criteria of admissibility}, Integral Equations Operator Theory
  \textbf{25(2)} (1996), 182--198.

\bibitem{Ho_Russell_1983}
L.F. Ho and D.L. Russell, \emph{{Admissible input elements for systems in
  Hilbert space and a Carleson measure criterion}}, {SIAM Journal of Control
  and Optimization} \textbf{21} (1983), 616--640.

\bibitem{Ho_Russell_1983_err}
\bysame, \emph{{Erratum: Admissible input elements for systems in Hilbert space
  and a Carleson measure criterion}}, {SIAM Journal of Control and
  Optimization} \textbf{21} (1983), 985--986.

\bibitem{Jacob_Partington_2004}
B.~Jacob and J.~R. Partington, \emph{Admissibility of control and observation
  operators for semigroups: A survey}, Current Trends in Operator Theory and
  its Applications (Joseph~A. Ball, J.~William Helton, Martin Klaus, and Leiba
  Rodman, eds.), Birkh{\"a}user Basel, Basel, 2004, pp.~199--221.

\bibitem{Jacob_Partington_Pott_2013}
B.~Jacob, J.~R. Partington, and S.~Pott, \emph{On {Laplace-Carleson} embedding
  theorems}, Journal of Functional Analysis \textbf{264} (2013), 783--814.

\bibitem{Jacob_Partington_Pott_2014}
\bysame, \emph{Applications of {Laplace-Carleson} embeddings to admissibility
  and controllability}, SIAM Journal of Control and Optimization \textbf{52}
  (2014), no.~2, 1299--1313.

\bibitem{Kapica_Zawiski_2022}
R.~Kapica and R.~Zawiski, \emph{Conditions for asymptotic stability of first order scalar differential-difference equation with complex coefficients}, ArXiv
  e-prints (2022), https://arxiv.org/abs/2204.08729v2.

\bibitem{Kappel_1986}
F.~Kappel, \emph{Semigroups and delay equations}, Semigroups, theory and
  applications. Vol. 2 (H.~Brezis, M.~Crandall, and F.~Kappel, eds.), Pitman
  Research Notes in Mathematics Series, vol. 152, Harlow: Longman Scientific
  and Technical, New York, 1986, pp.~136--176.

\bibitem{Khodja_et_al_2014}
F.~A. Khodja, C. Bouzidi, C. Dupaix and L. Maniar, \emph{Null controllability of retarded parabolic equations}, {Mathematical Control and Related Fields} \textbf{4(1)} (2014), 1--15

\bibitem{Khusainov_Pokojovy_Azizbayov_2013}
D. Ya. Khusainov, M. Pokojovy and E. I. Azizbayov, \emph{Classical Solvability for a Linear 1D Heat Equation with Constant Delay}, Konstanzer Schriften in Mathematik \textbf{316} (2013)

\bibitem{Koosis}
P.~{Koosis}, \emph{{Introduction to $H_p$ Spaces}}, 2nd ed., Cambridge Tracts
  in Mathematics, vol. 115, Cambridge University Press, Cambridge, UK, 2008.

\bibitem{Partington_1988}
J.~R. {Partington}, \emph{{An introduction to Hankel operators}}, London
  Mathematical Society Student Texts, vol.~13, Cambridge University Press,
  Cambridge, UK, 1988.

\bibitem{Partington_Zawiski_2019}
J.~R. {Partington} and R.~{Zawiski}, \emph{{Admissibility of state delay
  diagonal systems with one-dimensional input space}}, Complex Analysis and
  Operator Theory \textbf{13} (2019), 2463–2485.

\bibitem{Rudin_1987}
W.~{Rudin}, \emph{{Real and Complex Analysis}}, third ed., McGraw-Hill,
  Singapore, 1987.

\bibitem{Tucsnak_Weiss}
M.~{Tucsnak} and G.~{Weiss}, \emph{{Observation and Control for Operator
  Semigroups}}, Birkh{\"a}user Verlag AG, Basel, 2009.

\bibitem{Walton_Marshall_1984}
K.~Walton and J.~E. Marshall, \emph{Closed form solutions for time delay
  systems' cost functionals}, International Journal of Control \textbf{39}
  (1984), 1063--1071.

\bibitem{Weiss_1988}
G.~Weiss, \emph{{Admissible input elements for diagonal semigroups on
  {$l^2$}}}, {Systems and Control Letters} \textbf{10} (1988), 79--82.

\bibitem{Weiss_1999}
\bysame, \emph{A powerful generalization of the {Carleson} measure theorem?},
  Open problems in Mathematical Systems and Control Theory, Comm. Control
  Engrg., Springer, London, 1999, pp.~267--272.

\bibitem{Wu_1996}
J.~{Wu}, \emph{{Theory and Applications of Partial Functional Differential Equations}}, 
Springer-Verlag, New York, 1996.

\bibitem{Wynn_2010}
A.~Wynn, \emph{{$\alpha$-Admissibility of Observation Operators in Discrete and
  Continuous Time}}, {Complex Analysis and Operator Theory} \textbf{4} (2010),
  no.~1, 109--131.

\end{thebibliography}

\providecommand{\bysame}{\leavevmode\hbox to3em{\hrulefill}\thinspace}
\providecommand{\MR}{\relax\ifhmode\unskip\space\fi MR }
\providecommand{\MRhref}[2]{%
  \href{http://www.ams.org/mathscinet-getitem?mr=#1}{#2}
}
\providecommand{\href}[2]{#2}

\end{document}